\documentclass[11pt]{article}
\usepackage{latexsym,amsmath,amssymb}

\newif\ifdviwin

\usepackage[center]{titlesec}
\usepackage[english]{babel}
\usepackage{indentfirst}
\usepackage[mathscr]{eucal}
\usepackage{amssymb,amsmath,amsfonts}
\usepackage{fancybox,fancyhdr}
\usepackage{graphicx}
\usepackage[utf8]{inputenc}
\usepackage{float}
\usepackage{color}
\usepackage[pdftex]{hyperref}
\hypersetup{colorlinks=true,linkcolor=blue,citecolor=blue} 
\usepackage[export]{adjustbox}

\newif\ifdviwin

\dviwintrue

\def\H{\mathcal{H}}
\def\H{\H}

\def\m2r{\mathbb{M}^2\times\mathbb{R}}
\def\h2r{\mathbb{H}^2\times\mathbb{R}}

\let\infty=\infty \let\0=\emptyset  
\let\hat=\widehat

\let\parc=\partial
\let\varepsilon=\varepsilon

\def\cte.{\mathop{\rm cte.}\nolimits}

\def\cosh{\mathop{\rm cosh }\nolimits}
\def\tanh{\mathop{\rm tanh }\nolimits}

\def\R{\mathbb{R}}

\def\M{\mathbb{M}}

\def\H{\mathcal{H}}
\def\f{\mathfrak{f}}

\def\S{\mathbb{S}}

\def\m2r{\M^2\times\R}

\def\h2r{\mathbb{H}^2\times\R}
\def\s2r{\mathbb{S}^2\times\R}

\def\lH{\H_\lambda\text{-}\text{hypersurface}}
\def\lHn{\H_\lambda\text{-}\text{hypersurfaces}}

\def\sig{\Sigma}
\def\r3{\mathbb{R}^3}

\def\n1{_{n+1}}
\def\rnn{\mathbb{R}^{n+1}}
\def\sn{\mathbb{S}^n}

 \newtheorem{defi}{Definition}[section]
 \newtheorem{teo}[defi]{Theorem}
 \newtheorem{pro}[defi]{Proposition}
 
 \newtheorem{lem}[defi]{Lemma}

 \newenvironment{proof}{\rm \trivlist \item[\hskip \labelsep{\it
      Proof}:]}{\nopagebreak \hfill $\Box$ \endtrivlist}

\numberwithin{equation}{section}

\begin{document}


\begin{center}

\renewcommand{\thefootnote}{\,}
{\large \bf Invariant hypersurfaces with linear prescribed mean curvature
\footnote{\hspace{-.75cm}
}}\\
\vspace{0.5cm} { Antonio Bueno$^\dagger$, Irene Ortiz$^\ddagger$}\\
\end{center}
\vspace{.5cm}
$^\dagger$Departamento de Geometría y Topología, Universidad de Granada, E-18071 Granada, Spain. \\ \vspace{.3cm}
\emph{E-mail address:} jabueno@ugr.es\\
$^\ddagger$Departamento de Ciencias e Informática, Centro Universitario de la Defensa de San Javier, E-30729 Santiago de la Ribera, Spain. \\ \vspace{.3cm}
\emph{E-mail address:} irene.ortiz@cud.upct.es


\begin{abstract}
Our aim is to study invariant hypersurfaces immersed in the Euclidean space $\mathbb{R}^{n+1}$, whose mean curvature is given as a linear function in the unit sphere $\mathbb{S}^n$ depending on its Gauss map. These hypersurfaces are closely related with the theory of manifolds with density, since their weighted mean curvature in the sense of Gromov is constant. In this paper we obtain explicit parametrizations of constant curvature hypersurfaces, and also give a classification of rotationally invariant hypersurfaces.
\end{abstract}

\begin{center}
Contents

$\begin{array}{lr}
1.\hspace{.25cm} \text{Introduction} &\hspace{3cm}\hyperlink{page.1}{1}\\
2.\hspace{.25cm} \text{Constant curvature}\ \lHn&\hyperlink{page.3}{3}\\
3.\hspace{.25cm} \text{The phase plane of rotational}\ \lHn &\hyperlink{page.7}{7}\\
4.\hspace{.25cm} \text{Classification of rotational}\ \lHn &\hyperlink{page.8}{8}\\
5.\hspace{.25cm} \text{References} & \hyperlink{page.21}{21}
\end{array}
$
\end{center}

\section{\large Introduction}\label{intro}
\vspace{-.5cm}

Let us consider an oriented hypersurface $\sig$ immersed into $\rnn$ whose mean curvature is denoted by $H_\sig$ and its Gauss map by $\eta:\sig\rightarrow\S^n\subset\rnn$. Following \cite{BGM1}, given a function $\H\in C^1(\S^n)$, $\sig$ is said to be a hypersurface of \emph{prescribed mean curvature} $\H$ if 
\begin{equation}\label{prescribedMC}
H_\sig(p)=\H(\eta_p),
\end{equation}
for every point $p\in\sig$. Observe that when the prescribed function $\H$ is constant, $\sig$ is a hypersurface of constant mean curvature (CMC).

It is a classical problem in the Differential Geometry the study of hypersurfaces which are defined by means of a prescribed curvature function in terms of the Gauss map, being remarkable the Minkowski and Christoffel problems for ovaloids (\cite{Min,Chr}). In particular, when such prescribed function is the mean curvature, the hypersurfaces arising are the ones governed by \eqref{prescribedMC}. For them, the existence and uniqueness of ovaloids was studied, among others, by Alexandrov and Pogorelov in the '50s, \cite{Ale,Pog}, and more recently by Guan and Guan in \cite{GuGu}. Nevertheless, the global geometry of complete, non-compact hypersurfaces of prescribed mean curvature in $\rnn$ has been unexplored for general choices of $\H$ until recently. In this framework, the first author jointly with Gálvez and Mira have started to develop the \emph{global theory of hypersurfaces with prescribed mean curvature} in \cite{BGM1}, taking as a starting point the well-studied global theory of CMC hypersurfaces in $\rnn$. The same authors have also studied rotational hypersurfaces in $\rnn$, getting a Delaunay-type classification result and several examples of rotational hypersurfaces with further symmetries and topological properties (see \cite{BGM2}). For prescribed mean curvature surfaces in $\r3$, see \cite{Bue1} for the resolution of the Björling problem and \cite{Bue2} for the obtention of half-space theorems properly immersed surfaces.


Our objective in this paper is to further investigate the geometry of complete hypersurfaces of prescribed mean curvature for a relevant choice of the prescribed function. In particular, let us consider $\H\in C^1(\S^n)$ a linear function, that is, 
\[
\H(x)=a\langle x,v\rangle+\lambda
\]
for every $x\in\S^n$, where $a,\lambda\in\R$ and $v$ is a unit vector called the \emph{density vector}. Note that if $a=0$ we are studying hypersurfaces with constant mean curvature equal to $\lambda$. Moreover, if $\lambda=0$, we are studying self-translating solitons of the mean curvature flow, case which is widely studied in the literature (see e.g.  \cite{CSS,HuSi,Ilm,MSHS,SpXi} and references therein). Therefore, we will assume that $a$ and $\lambda$ are not null in order to avoid the trivial cases. Furthermore, after a homothety of factor $1/a$ in $\rnn$, we can get $a=1$ without loss of generality. Bearing in mind these considerations, we focus on the following class of hypersurfaces.

\begin{defi}
An immersed, oriented hypersurface $\sig$ in $\rnn$ is an $\H_\lambda$-hypersurface if its mean curvature function $H_\sig$ is given by
\begin{equation}\label{defilambdasup}
H_\sig(p)=\H_\lambda(\eta_p)=\langle\eta_p,v\rangle+\lambda, \quad \forall p\in\sig.
\end{equation}
\end{defi}
See that if $\sig$ is an $\lH$ with Gauss map $\eta$, then $\sig$ with the opposite orientation $-\eta$ is trivially a $\H_{-\lambda}$-hypersurface. Thus, up to a change of the orientation, we assume $\lambda>0$. 

The relevance of the class of $\lHn$ lies in the fact that they satisfy some characterizations which are closely related to the theory of manifolds with density. Firstly, following Gromov \cite{Gro}, for an oriented hypersurface $\sig$ in $\R^{n+1}$ with respect to the density $e^\phi\in C^1(\R^{n+1})$, the \emph{weighted mean curvature} $H_\phi$ of $\sig$ is defined by
\begin{equation}\label{weightedMC}
H_\phi:=H_\sig-\langle\eta,\nabla\phi\rangle,
\end{equation}
where $\nabla$ is the gradient operator of $M$. Note that when the density is $\phi_v(x)=\langle x,v\rangle$, by using \eqref{defilambdasup} and \eqref{weightedMC} it follows that $\sig$ is an $\lH$ if and only if $H_{\phi_v}=\lambda$. In particular, as pointed out by Ilmanen \cite{Ilm}, self-translating solitons are \emph{weighted minimal}, i.e. $H_{\phi_v}=0$. On the other hand, although hypersurfaces of prescribed mean curvature do not come in general associated to a variational problem, the $\lHn$ do. To be more specific, consider any measurable set $\Omega\subset\rnn$ having as boundary $\sig=\partial\Omega$ and inward unit normal $\eta$ along $\sig$. Then, the \emph{weighted area and volume} of $\Omega$ with respect to the density $\phi_v$ are given respectively by
$$
A_{\phi_v}(\sig):=\int_\sig e^{\phi_v} d\sig,\hspace{.5cm} V_{\phi_v}(\Omega):=\int_\Omega e^{\phi_v} dV,
$$
where $d\sig$ and $dV$ are the usual area and volume elements in $\rnn$. So, in \cite{BCMR} it is proved that $\sig$ has constant weighted mean curvature equal to $\lambda$ if and only if $\sig$ is a critical point under compactly supported variations of the functional $J_{\phi_v}$, where
$$
J_{\phi_v}:=A_{\phi_v}-\lambda V_{\phi_v}.
$$
Finally, observe that if $f:\sig\rightarrow\R^{n+1}$ is an $\lH$, the family of translations of $f$ in the $v$ direction given by $F(p,t)=f(p)+tv$ is the solution of the geometric flow  
\begin{equation}\label{geoflo}
\left(\frac{\parc F}{\parc t}\right)^{\bot}=(H_\sig-\lambda)\eta,
\end{equation}
which corresponds to the mean curvature flow with a constant forcing term, that is, $f$ is a self-translating soliton of the geometric flow \eqref{geoflo}. This flow already appeared in the context of studying the \emph{volume preserving mean curvature flow}, introduced by Huisken \cite{Hui}. 


Throughout this work we focus our attention on $\lHn$ which are \emph{invariant} under the flow of an $(n-1)$-group of translations and the isometric $SO(n)$-action of rotations that pointwise fixes a straight line. The first group of isometries generates \emph{cylindrical flat hypersurfaces}, while the second one corresponds to \emph{rotational hypersurfaces}. These isometries and the symmetries inherited by the invariant $\lHn$ are induced to Equation \eqref{defilambdasup} easing the treatment of its solutions. We must emphasize that, although the authors already defined the class of immersed $\lHn$ in \cite{BGM1}, the classification of neither cylindrical nor rotational $\lHn$ in \cite{BGM2} was covered.

We next detail the organization of the paper. In Section \ref{constantcurv} we study complete $\lHn$ that have constant curvature. By classical theorems of Liebmann, Hilbert and Hartman-Nirenberg, any such $\lH$ must be flat, hence invariant by an $(n-1)$-group of translations and described as the riemannian product $\alpha\times\R^{n-1}$, where $\alpha$ is a plane curve called the \emph{base curve}. This product structure allows us to relate the condition of being an $\lH$ with the geometry of $\alpha$. Indeed, the curvature $\kappa_\alpha$ is, essentially, the mean curvature of $\alpha\times\R^{n-1}$. In Theorem \ref{clasificacioncurvaturacte} we classify such $\H_\lambda$-hypersurfaces by giving explicit parametrizations of the base curve.

Later, in Section \ref{properties} we introduce the phase plane for the study of rotational $\lHn$. In particular, we treat the ODE that the profile curve of a rotational $\lH$ satisfies as a non-linear autonomous system since the qualitative study of the solutions of this system will be carried out by a phase plane analysis, as the first author did jointly with Gálvez and Mira in \cite{BGM2}. Finally, in Section \ref{rot} we give a complete classification of rotational $\lHn$ intersecting the axis of rotation in Theorem \ref{Classification1} and non-intersecting such axis in Theorem \ref{Classification2}. To get such results we develop along this section a discussion depending on the value of $\lambda$, namely $\lambda>1,\ \lambda=1$ and $\lambda<1$. 

\section{\large Constant curvature $\H_\lambda$-hypersurfaces}\label{constantcurv}
\vspace{-.5cm}

The aim of this section is to obtain a classification result for complete $\lHn$ with constant curvature. By classical theorems of Liebmann, Hilbert, and Hartman-Nirenberg, any such $\lH$ must be flat, hence invariant by an $(n-1)$-parameter group of translations $\mathcal{G}_{a_1,...,a_{n-1}}=\{F_{t_1,...,t_{n-1}};\ t_i\in\R\}$ where $a_i\in\rnn$ with $i=1,...,n-1$, are linearly independent and $F_{t_1,...,t_{n-1}}(p)=p+\sum_{i=1}^{n-1}t_i a_i$, for every $p\in\rnn$. Any $\lH$ invariant by such a group is called a \emph{cylindrical flat $\H_\lambda$-hypersurface}, and the directions $a_1,...,a_{n-1}$ are known as \emph{ruling directions}. 

For cylindrical flat hypersurfaces having as rulings $a_1,...,a_{n-1}$, it is known that a global parametri-
zation is given by 
$$
\psi(s,t_1,...,t_{n-1})=\alpha(s)+\sum_{i=1}^{n-1}t_i a_i,
$$
where $\alpha$ is a curve, called the \emph{base curve}, contained in a 2-dimensional plane $\Pi$ orthogonal to the vector space $\mathrm{Lin}\langle a_1,...,a_{n-1}\rangle$. Henceforth, we will denote a cylindrical flat $\lH$ by $\sig_\alpha:=\alpha\times\R^{n-1}$, where $\R^{n-1}$ stands for the orthogonal complement of $\Pi$. From this parametrization we obtain that $\sig_\alpha$ has, at most, two different principal curvatures: one given by the curvature of $\alpha$, $\kappa_\alpha$, and the $n-1$ remaining being identically zero. Since the mean curvature $H_{\sig_\alpha}$ of $\sig_\alpha$ is given as the mean of its principal curvatures, it follows from Equation \eqref{defilambdasup} that $\kappa_\alpha$ satisfies
\begin{equation}\label{curvaturaalpha}
\kappa_\alpha=n H_{\sig_\alpha}=n(\langle\textbf{n}_\alpha,v\rangle+\lambda),
\end{equation}
where $\textbf{n}_\alpha:= J\alpha'$ is the positively oriented, unit normal of $\alpha$ in $\Pi$. We must emphasize that there is no a priori relation between the density vector $v$ and the ruling directions $a_i$.

It is immediate that if $\Pi^\bot$ and $v$ are parallel, then Equation \eqref{curvaturaalpha} implies that $\kappa_\alpha=\lambda n$ is constant, and thus $\alpha$ is a straight line or a circle in $\Pi$ of radius $1/(\lambda n)$. Hence, \emph{hyperplanes and right circular cylinders are the only $\H_\lambda$-hypersurfaces whose rulings are parallel to the density vector}. Another particular but important case appears when $\lambda=0$, that is, for translating solitons. It is known that if $v$ and $\Pi^\bot$ are orthogonal, the cylindrical translating solitons are hyperplanes generated by $\Pi^\bot$ and $v$, and \emph{grim reaper} cylinders.

After a change of Euclidean coordinates, we suppose that the plane $\Pi$ is the one generated by the vectors $e_1$ and $e\n1$, and after a rotation around $e_1$ we suppose that the density vector $v$ has coordinates $v=(0,v_2,...,v\n1)$. Moreover, we can assume that $v\n1\neq 0$; otherwise $v$ and the ruling directions $\Pi^\bot$ are parallel and $\alpha$ is a straight line or a circle in $\Pi$. Assume that $\alpha(s)=(x(s),0,...,0,z(s))$ is arc-length parametrized, that is $x'(s)=\cos\theta(s),\ z'(s)=\sin\theta(s)$ where the function $\theta(s)$ is the angle between $\alpha'(s)$ and the $e_1$-direction. Since the curvature $\kappa_\alpha(s)$ is given by $\theta'(s)$, Equation \eqref{curvaturaalpha} is equivalent to the following
\begin{equation}\label{sistemadiferencial}
\left\lbrace\begin{array}{l}
\vspace{.25cm}
x'(s)=\cos\theta(s)\\
\vspace{.25cm}
z'(s)=\sin\theta(s)\\
\theta'(s)=n\big(v\n1\cos\theta(s)+\lambda\big).
\end{array}
\right.
\end{equation}
We point out that for certain values of $\lambda$, system \eqref{sistemadiferencial} has trivial solutions. Indeed, suppose that $\lambda\in [-v\n1,v\n1]$ and let $\theta_0$ be such that $\cos\theta_0=-\lambda/v\n1$. Then, the straight line parametrized by $x(s)=(\cos\theta_0) s,\ z(s)=(\sin\theta_0) s,\ \theta(s)=\theta_0$ solves \eqref{sistemadiferencial}. Thus, by uniqueness of the ODE \eqref{sistemadiferencial}, \emph{if $\alpha$ has curvature vanishing at some point, it is a straight line}.

Now we solve \eqref{sistemadiferencial} for the case that $v\n1\neq 0$ and $\theta'(s)\neq 0$. Integrating the last equation we obtain the explicit expression of the function $\theta(s)$, depending on $\lambda$ and $v\n1$:

$$
\theta(s)=\left\lbrace\begin{array}{ll}
\vspace{.25cm}
2\arctan\left(\sqrt{\frac{\lambda+v\n1}{\lambda-v\n1}}\tan\left(\frac{n}{2}\sqrt{\lambda^2-v\n1^2}s\right)\right) & \hspace{.5cm}\mathrm{if}\ \lambda>v\n1\\
\vspace{.25cm}
2\arctan(nv\n1s)& \hspace{.5cm}\mathrm{if}\ \lambda=v\n1\\
\vspace{.25cm}
2\arctan\left(\sqrt{\frac{v\n1+\lambda}{v\n1-\lambda}}\tanh\left(\frac{n}{2}\sqrt{v\n1^2-\lambda^2}s\right)\right) & \hspace{.5cm}\mathrm{if}\ \lambda<v\n1,\ \mathrm{and}\ \theta(0)=0\\
\vspace{.25cm}
2\mathrm{arccotg}\left(\sqrt{\frac{v\n1-\lambda}{v\n1+\lambda}}\tanh\left(\frac{n}{2}\sqrt{v\n1^2-\lambda^2}s\right)\right) & \hspace{.5cm}\mathrm{if}\ \lambda<v\n1,\ \mathrm{and}\ \theta(0)=\pi.
\end{array}\right.
$$

Since $x'(s)=\cos\theta(s)$ and $z'(s)=\sin\theta(s)$, explicit integration yields the following classification result:

\begin{teo}\label{clasificacioncurvaturacte}
Up to vertical translations, the coordinates of the base curve of a cylindrical flat $\H_\lambda$-hypersurface $\sig_\alpha$ are classified as follows:
\end{teo}

\begin{itemize}
\item[1.] Case $\lambda>v\n1$. The explicit coordinates of $\alpha(s)$ are:
$$
\begin{array}{l}
\vspace{.25cm}x(s)=-\lambda s+\frac{2}{n}\arctan\left(\sqrt{\frac{\lambda+v\n1}{\lambda-v\n1}}\tan\left(\frac{n}{2}\sqrt{\lambda^2-v\n1^2}s\right)\right),\\
z(s)=\frac{1}{n}\log\left(\lambda-\cos\left(n\sqrt{\lambda^2-v\n1^2}s\right)\right).
\end{array}	
$$
The angle function $\theta(s)$ is periodic, the $x(s)$-coordinate is unbounded and the $z(s)$-coordinate is also periodic. The curve $\alpha(s)$ self-intersects infinitely many times. 

\begin{figure}[H]
\centering
\includegraphics[width=.5\textwidth]{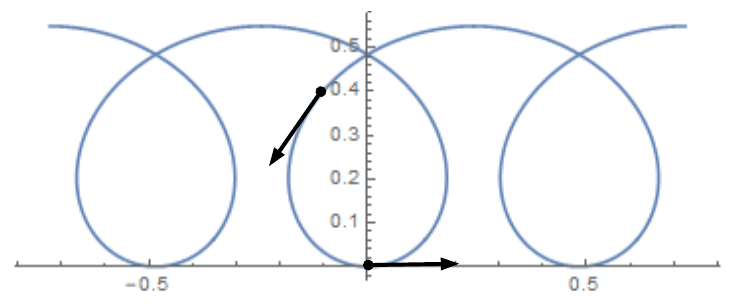}  
\caption{The profile curve for the case $\lambda>v\n1$. Here, $n=2$, $v\n1=1$ and $\lambda=2$.}
\label{lmayor1}
\end{figure}

\item[2.] Case $\lambda=v\n1$. 

\begin{itemize}
\item[2.1.] Either $\alpha(s)$ is a horizontal straight line parametrized by $x(s)=-s,\ z(s)=c_0,\ c_0\in\R,\ \theta(s)=\pi$, or

\item[2.2.] its explicit coordinates are
$$
\begin{array}{l}
\vspace{.25cm}x(s)=-s+\frac{n}{2}\arctan(ns),\\
z(s)=\frac{1}{n}\log(1+n^2s^2).
\end{array}
$$
The image of the angle function $\theta(s)$ in the circle $\S^1$ is $\S^1-\{(0,-1)\}$. The $z(s)$-coordinate decreases until reaching a minimum and then increases, and $\alpha(s)$ has a self-intersection.
\end{itemize}

\begin{figure}[H]
\centering
\includegraphics[width=.5\textwidth]{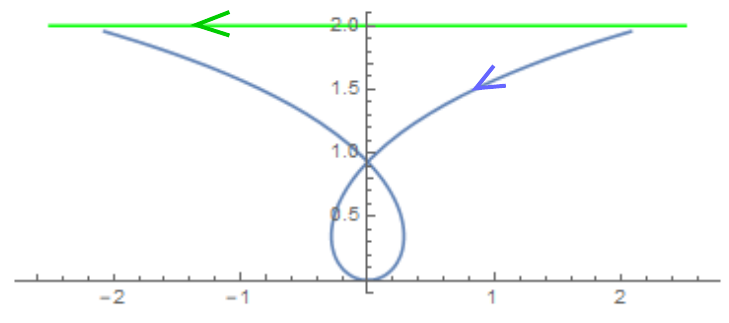}  
\caption{The profile curves for the case $\lambda=v\n1$. Here, $n=2$, $v\n1=1$ and $\lambda=1$.}
\label{ligual1}
\end{figure}

\item[3.] Case $\lambda<v\n1$.

\begin{itemize}
\item[3.1.] Either $\alpha(s)$ is a straight line parametrized by $x(s)=(\cos\theta_0)s,\ z(s)=\pm (\sin\theta_0)s,\ \theta(s)=\theta_0$, where $\theta_0$ is such that $\lambda+v\n1\cos\theta_0=0$, or

\item[3.2.] if $\theta(0)=0$, its explicit coordinates are
$$
\begin{array}{l}
\vspace{.25cm}x(s)=-\lambda s+\frac{2}{nv\n1}\arctan\left(\sqrt{\frac{v\n1+\lambda}{v\n1-\lambda}}\tanh\left(\frac{n}{2}\sqrt{v\n1^2-\lambda^2}s\right)\right),\\
z(s)=\frac{1}{nv\n1}\log\left(-\lambda+\cosh\left(n\sqrt{v\n1^2-\lambda^2}s\right)\right).
\end{array}
$$
In this case, $\alpha(s)$ has a self-intersection.

\item[3.3.] if $\theta(0)=\pi$, its explicit coordinates are
$$
\begin{array}{l}
\vspace{.25cm}x(s)=-\lambda s-\frac{2}{nv\n1}\arctan\left(\sqrt{\frac{v\n1+\lambda}{v\n1-\lambda}}\tanh\left(\frac{n}{2}\sqrt{v\n1^2-\lambda^2}s\right)\right),\\
z(s)=\frac{1}{nv\n1}\log\left(\lambda+\cosh\left(n\sqrt{v\n1^2-\lambda^2}s\right)\right).
\end{array}
$$
In this case, $\alpha(s)$ is a graph hence it is embedded.
\end{itemize}
In the two latter cases, the image of the angle function $\theta(s)$ of each curve is a connected arc in $\S^1$ whose endpoints are $(\cos\theta_0,\pm\sin\theta_0)$.

\begin{figure}[H]
\centering
\includegraphics[width=.6\textwidth]{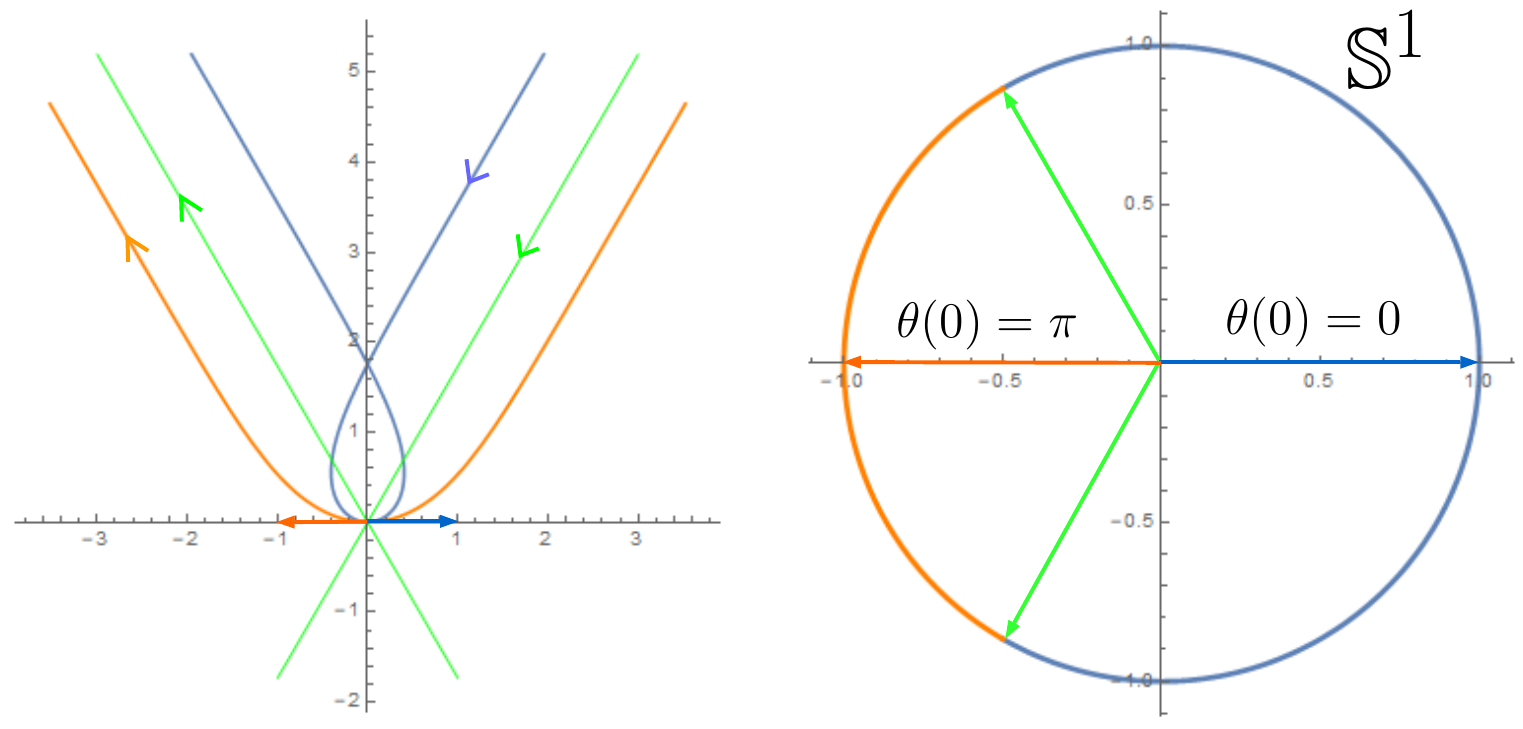}  
\caption{Left: the profile curves for the case $\lambda<v\n1$. In blue, the case $\theta(0)=0$; in orange, the case $\theta(0)=\pi$. Here, $n=2$, $v\n1=1$ and $\lambda=1/2$. Right: the values of $\theta(s)$ in $\S^1$ of each curve.}
\label{lmenor1}
\end{figure}
\end{itemize}


\section{\large The phase plane of rotational $\H_\lambda$-hypersurfaces}\label{properties}
\vspace{-.5cm}
This section is devoted to compile the main features of the phase plane for the study of rotational $\lHn$. To do so we follow \cite{BGM2}, where the phase plane was used to study rotational hypersurfaces of prescribed mean curvature given by Equation \eqref{prescribedMC}. 

%
%

%

Let us fix the notation. Firstly, observe that in contrast with cylindrical $\lHn$, where there was$C(\frac{n-1}{\lambda n})$ no a priori relation between the density vector and the ruling directions, for a rotational $\lH$ the density vector and the rotation axis must be parallel \cite[Proposition 4.3]{Lop}. Thus, after a change of Euclidean coordinates, we suppose that the density vector $v$ in Equation \eqref{defilambdasup} is $e\n1$. Then, we consider $\sig$ the rotational $\lH$ generated as the orbit of an arc-length parametrized curve
$$
\alpha(s)=(x(s),0,...,0,z(s)),\hspace{.5cm} s\in I\subset\R,
$$
contained in the plane $[e_1,e\n1]$ generated by the vectors $e_1$ and $e\n1$, under the isometric $SO(n)$-action of rotations that leave pointwise fixed the $x\n1$-axis. From now on, we will denote the coordinates of $\alpha(s)$ simply by $(x(s),z(s))$ and omit the dependence of the variable $s$, unless necessary. Note that the unit normal of $\alpha$ in $[e_1,e\n1]$, given by $\textbf{n}_\alpha=J\alpha'=(-z',x')$, induces a unit normal to $\sig$ by just rotating $\textbf{n}_\alpha$ around the $x\n1$-axis, and the principal curvatures of $\sig$ with respect to this unit normal are given by
$$
\kappa_1=\kappa_\alpha=x'z''-x''z',\hspace{.5cm} \kappa_2=\cdots=\kappa_n=\frac{z'}{x}.
$$
Consequently, the mean curvature $H_\sig$ of $\sig$, which satisfies \eqref{defilambdasup}, is related with $x$ and $z$ by
\begin{equation}\label{odemedia}
nH_\sig=n(x'+\lambda)=x'z''-x''z'+(n-1)\frac{z'}{x}.
\end{equation} 

As $\alpha$ is arc-length parametrized, it follows that $x$ is a solution of the second order autonomous ODE:
\begin{equation}\label{odex}
x''=(n-1)\frac{1-x'^2}{x}-n\varepsilon(x'+\lambda)\sqrt{1-x'^2}, \hspace{1cm} \varepsilon ={\rm sign}(z'),
\end{equation}
on every subinterval $J\subset I$ where $z'(s)\neq 0$ for all $s\in J$. Here, the value $\varepsilon$ denotes whether the height of $\alpha$ is increasing (when $\varepsilon=1$) or decreasing (when $\varepsilon=-1$).

After the change $x'=y$, \eqref{odex} transforms into the first order autonomous system
\begin{equation}\label{1ordersys}
\left(\begin{array}{c}
x\\
y
\end{array}\right)'=\left(\begin{array}{c}
y\\
(n-1)\frac{\displaystyle{1-y^2}}{\displaystyle{x}}-n\varepsilon(y+\lambda)\sqrt{1-y^2}
\end{array}\right).
\end{equation}

The \emph{phase plane} is defined as the half-strip $\Theta_\varepsilon:=(0,\infty)\times(-1,1)$, with coordinates $(x,y)$ denoting, respectively, the distance to the axis of rotation and the \emph{angle function} of $\sig$. The \emph{orbits} are the solutions $\gamma(s)=(x(s),y(s))$ of system \eqref{1ordersys}. Both the local and global behavior of an orbit in $\Theta_\varepsilon$ are strongly influenced by the underlying geometric properties of Equation \eqref{defilambdasup}. For example, since the profile curve $\alpha$ of a rotational $\lH$ only intersects the axis of rotation orthogonally, see e.g. \cite[Theorem 4.1, pp. 13-14]{BGM2}, an orbit in $\Theta_\varepsilon$ cannot converge to a point $(x_0,y_0)$ with $x_0=0,\ y_0\in (-1,1)$. 

Next, we highlight some consequences of the study of the phase plane carried out in Section 2 in \cite{BGM2} adapted to our particular case.

\begin{lem}\label{resumenfases}
For each $\lambda>0$:
\begin{itemize}
\item[1.] There is a unique equilibrium of \eqref{1ordersys} in $\Theta_1$ given by $e_0:=\left(\frac{n-1}{\lambda n},0\right)$. This equilibrium generates the constant mean curvature, flat cylinder of radius $\frac{n-1}{\lambda n}$ and vertical rulings. 

\item[2.] The Cauchy problem associated to system \eqref{1ordersys} for the initial condition $(x_0,y_0)\in\Theta_\varepsilon$ has local existence and uniqueness. Consequently, the orbits provide a foliation of regular, proper, $C^1$ curves of $\Theta_\varepsilon-\{e_0\}$, and two distinct orbits cannot intersect in $\Theta_\varepsilon$. Moreover, by uniqueness of the Cauchy problem \eqref{1ordersys}, if an orbit $\gamma(s)$ converges to $e_0$, the value of the parameter $s$ goes to $\pm\infty$.

\item[3.] The points of $\alpha$ with $\kappa_\alpha=0$ are the ones where $y'=0$. They are located in $\Gamma_\varepsilon:=\Theta_\varepsilon\cap\{x=\Gamma_\varepsilon(y)\}$, where 
\begin{equation}\label{curvagamma}
\Gamma_{\varepsilon}(y)=\frac{(n-1)\sqrt{1-y^2}}{n \varepsilon(y+\lambda)},
\end{equation}
and $\varepsilon(y+\lambda)>0$.

\item[4.] The axis $y=0$ and $\Gamma_\varepsilon$ divide $\Theta_\varepsilon$ into connected components where the coordinate functions of an orbit $(x(s),y(s))$ are monotonous. Thus, at each of these monotonicity regions, the motion of an orbit is uniquely determined.


\item[5.] If an orbit $(x(s),y(s))$ intersects $\Gamma_\varepsilon$, the function $y(s)$ has a local extremum; if an orbit intersects the axis $y=0$, it does orthogonally. 
%
\end{itemize}
\end{lem}

Finally, recall that system \eqref{1ordersys} has a singularity for the values $x_0=0,\ y_0=\pm 1$, hence we cannot ensure the existence of a rotational $\lH$ intersecting orthogonally the axis of rotation by solving the Cauchy problem with this initial data. However, we can guarantee the existence of such a rotational $\lH$ by solving the Dirichlet problem over a small-enough domain, see \cite[Corollary 1]{Mar}. Now, Corollary 2.4 in \cite{BGM2} has the following implication in our phase plane study:

\begin{lem}\label{existorbitaext}
Let $\varepsilon,\delta\in\{-1,1\}$ be such that $\varepsilon(\delta+\lambda)>0$. Then, there exists a unique orbit in $\Theta_\varepsilon$ that has $(0,\delta)\in\overline{\Theta_\varepsilon}$ as an endpoint. There is no such an orbit in $\Theta_{-\varepsilon}$.
\end{lem}

%

\section{\large Classification of rotational $\H_\lambda$-hypersurfaces}\label{rot}
\vspace{-.5cm}

Throughout this section we classify rotational $\lHn$ depending on the value of $\lambda$. As a first approach to arise such a classification, we must mention a technical, useful in the later, result which establishes that no closed examples exist in the class of immersed $\lHn$. In particular, the case $n=2$ was originally compiled in López \cite{Lop} and its proof can be  easily extended to any dimension.

\begin{lem}\label{noclosed}
There do not exist closed $\lHn$.
\end{lem}

At this point, we are going to study the aforementioned classification by analyzing the qualitative properties of system \eqref{1ordersys}, most of them already studied in the previous section. To this end, it is useful to study its \emph{linearized} system at the unique equilibrium $e_0=\big(\frac{n-1}{\lambda n},0\big)$. In particular, the linearized of \eqref{1ordersys} at $e_0$ is given by
\begin{equation}
\left(\begin{matrix}
0&1\\
\displaystyle{-\frac{n^2\lambda^2}{n-1}}& -n
\end{matrix}\right),
\end{equation}
whose eigenvalues are
$$
\mu_1=\frac{-n+ n\sqrt{1-\displaystyle{\frac{4\lambda^2}{n-1}}}}{2},\hspace{1cm}\text{and} \hspace{1cm} \mu_2=\frac{-n- n\sqrt{1-\displaystyle{\frac{4\lambda^2}{n-1}}}}{2}.
$$

Standard theory of non-linear autonomous systems enables us to summarize the possible beha-
viors of a solution around the equilibrium $e_0$:
\begin{itemize}
\item if $\lambda>\sqrt{n-1}/2$, then $\mu_1$ and $\mu_2$ are complex conjugate with negative real part. Thus, $e_0$ has an \emph{inward spiral} structure, and every orbit close enough to $e_0$ converges asymptotically to it spiraling around infinitely many times.

\item if $\lambda=\sqrt{n-1}/2$, then $\mu_1=\mu_2$ and they are real and negative, with only one eigenvector. Thus, $e_0$ is an asymptotically stable improper node, and every orbit close enough to $e_0$ converges asymptotically to it, maybe spiraling around a finite number of times.

\item if $\lambda\in (0,\sqrt{n-1}/2)$, then $\mu_1$ and $\mu_2$ are different, real and negative. Thus, $e_0$ is an asymptotically stable node and has a \emph{sink} structure, and every orbit close enough to $e_0$ converges asymptotically to it \emph{directly}, i.e. without spiraling around.
\end{itemize}

\begin{figure}[H]
\centering
\includegraphics[width=.7\textwidth]{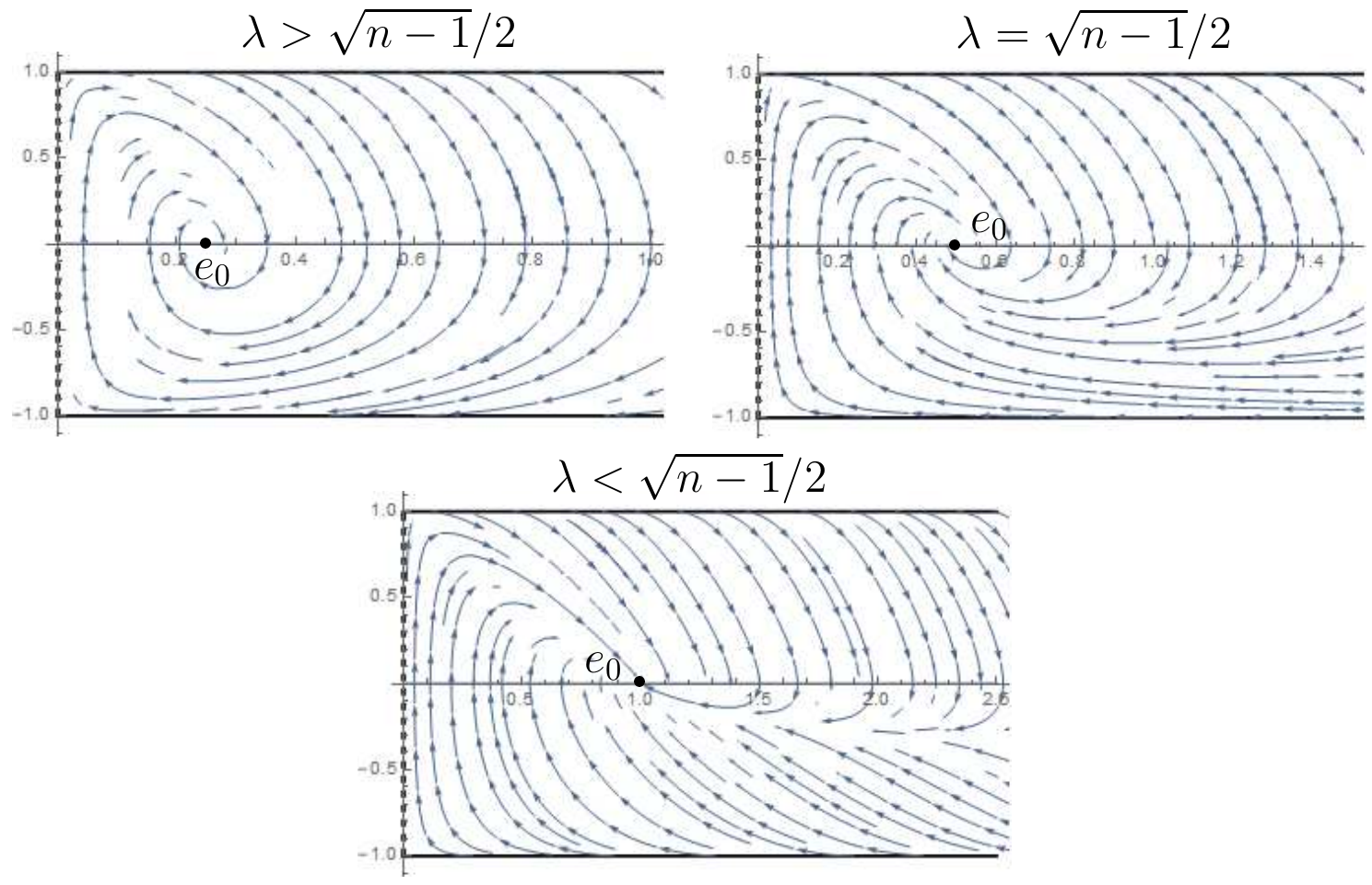}
\caption{The linearized of system \eqref{1ordersys} depending on the values of $\lambda>0$ and the behavior of its orbits.}
\label{linearizados}
\end{figure}

We now analyze the rotational $\lHn$ in $\rnn$ by distinguishing three possibilities for $\lambda$: $\lambda>1$, $\lambda=1$ and $\lambda<1$. These three cases will deeply influence the global behavior of the orbits in each phase plane. Additionally, in our discussion we take into account if such hypersurfaces intersect orthogonally the axis of rotation or not. 

\begin{center}
{\Large \textbf{\underline{Case $\lambda>1$}}}	
\end{center}
\vspace{-.25cm}

Let us assume $\lambda>1$. On the one hand, for $\varepsilon=1$, the curve $\Gamma_1$ given by Equation \eqref{curvagamma} is a compact, connected arc in $\Theta_1$ joining the points $(0,1)$ and $(0,-1)$. In order to study the monotonicity regions in $\Theta_1$, let us consider an arc-length parametrized curve
$\alpha(s)=(x(s),z(s))$ satisfying \eqref{odex} and $\gamma(s)$ the corresponding orbit that solves \eqref{1ordersys}. Combining items $\textit{3}$ and $\textit{4}$ in Lemma \ref{resumenfases} we can ensure that in $\Theta_1$ there are four monotonicity regions which will be called $\Lambda_1,...,\Lambda_4$, respectively (see Figure \ref{lmayor1}, left). Moreover, if the orbit $\gamma$ is contained in $\Lambda_1\cup\Lambda_2$, it corresponds to points of $\alpha$ with positive geodesic curvature, whereas, if on the contrary, $\gamma$ is contained in $\Lambda_3\cup\Lambda_4$, it corresponds to points of $\alpha$ with negative geodesic curvature.

On the other hand, for $\varepsilon=-1$, the curve $\Gamma_{-1}$ does not exist in $\Theta_{-1}$, and so there are only two monotonicity regions in $\Theta_{-1}$ called $\Lambda_+$ and $\Lambda_-$ (see Figure \ref{lmayor1}, right). In this case both regions correspond to points of $\alpha$ with positive geodesic curvature.

%

\begin{figure}[H]
\centering
\includegraphics[width=.9\textwidth]{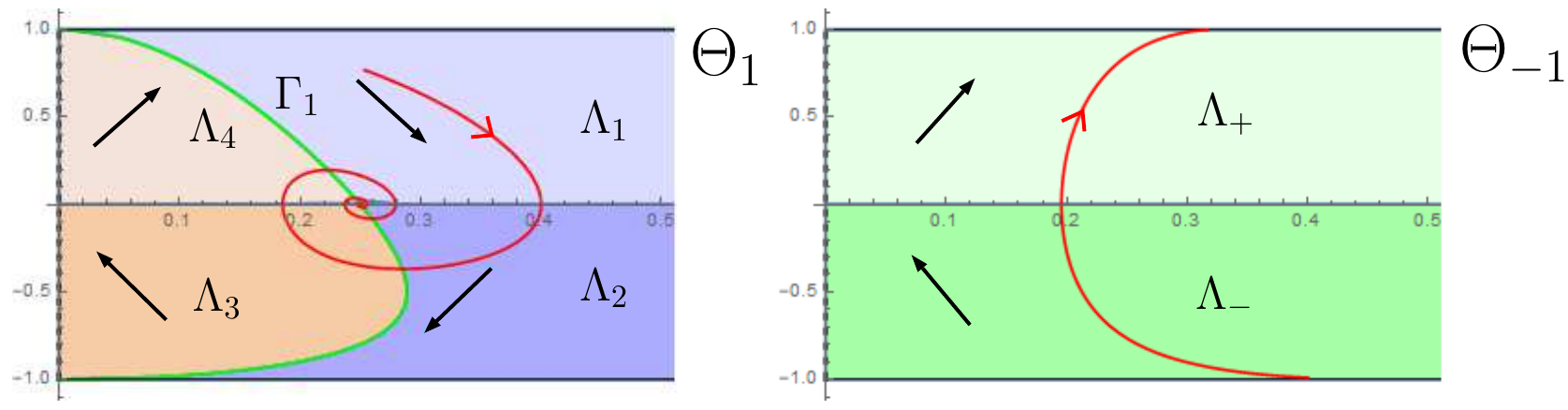}
\caption{The phase planes $\Theta_\varepsilon,\ \varepsilon=\pm1$ for $\lambda>1$, their monotonicity regions and two orbits following the motion at each monotonicity region.}
\label{lmayor1}
\end{figure}

Our first goal is to describe the rotational $\lHn$ that intersect orthogonally the axis of rotation. By Lemma \ref{existorbitaext} there is an orbit $\gamma_+(s)$ in $\Theta_1$ having $(0,1)$ as endpoint, and after a translation in $s$ we can suppose that $\gamma_+(0)=(0,1)$. This orbit generates an arc-length parametrized curve $\alpha_+(s)=(x_+(s),z_+(s))$ that intersects orthogonally the axis of rotation at the instant $s=0$. Since $\lambda>1$, by ODE \eqref{odemedia} we see that $z''_+(0)>0$ and so $z_+(s)$ has a minimum at $s=0$. As a matter of fact, for $s>0$ close enough to $s=0$ we have $z_+'(s)>0$ which implies that $x_+''(s)<0$. In particular, the geodesic curvature $\kappa_{\alpha_+}(s)$ of $\alpha_+$ is positive and so the orbit $\gamma_+(s)$ is strictly contained in the region $\Lambda_1$ for $s>0$ close enough to $s=0$. See Figure \ref{saleneje} where the orbit $\gamma_+$ and the curve $\alpha_+$ are ploted in red.

Once again, by Lemma \ref{existorbitaext} there is an orbit $\gamma_-(s)$ in $\Theta_1$ with $(0,-1)$ as endpoint. Such an orbit also generates an arc-length parametrized curve $\alpha_-(s)=(x_-(s),z_-(s))$ that intersects orthogonally the axis of rotation at $s=0$. A similar discussion as above yields that $z_-''(0)<0$ and so $z_-(s)$ has a maximum at $s=0$. Thus, for $s<0$ we have $z_-'(s)>0$ which implies again that $x_-''(s)<0$. This time, $\gamma_-(s)$ is strictly contained in the region $\Lambda_2$ for $s<0$ close enough to $s=0$. See Figure \ref{saleneje} where the orbit $\gamma_-$ and the curve $\alpha_-$ are ploted in orange.

\begin{figure}[H]
\centering
\includegraphics[width=.9\textwidth]{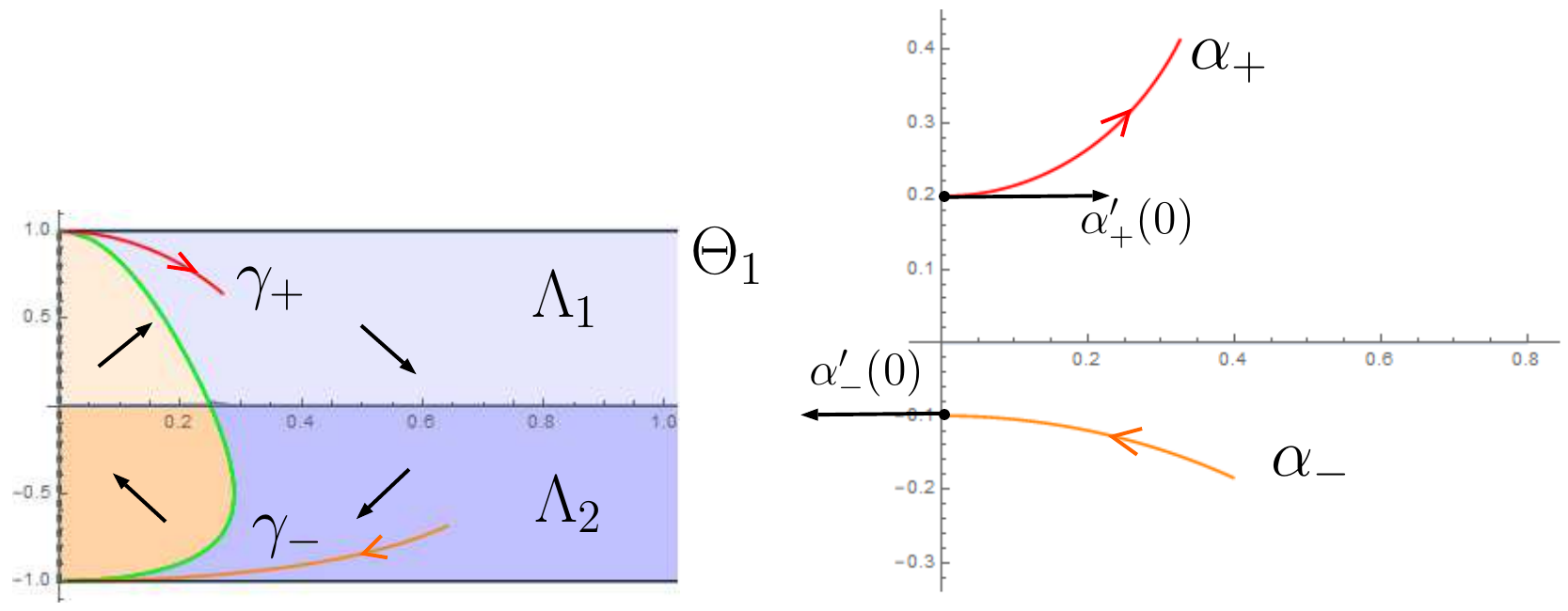}
\caption{Left: the phase plane $\Theta_1$ and the orbits $\gamma_+$ and $\gamma_-$. Right: the corresponding arc-length parametrized curves $\alpha_+$ and $\alpha_-$.}
\label{saleneje}
\end{figure}


Let us study in more detail the behavior of both orbits $\gamma_+$ and $\gamma_-$ in $\Theta_1$.

\begin{pro}\label{contradicecomparacioncurvaturamedia}	
Let us consider the orbits $\gamma_+$ and $\gamma_-$ in the phase plane $\Theta_1$ as above. Then:
	\begin{itemize}
		\item[1.] The orbit $\gamma_+(s)$ cannot stay forever in $\Lambda_1$. Moreover, it converges orthogonally to a point $(x_+,0)$ with $x_+\geq\frac{n-1}{\lambda n}$, which can be either the equilibrium $e_0$ with the parameter $s\rightarrow\infty$, or a finite point reaching it at some finite instant $s_+>0$.
		\item[2.] The orbit $\gamma_-(s)$ cannot stay forever in $\Lambda_2$.
		Moreover, it intersects orthogonally the axis $y=0$ at a point $(x_-,0)$ with $x_->\frac{n-1}{\lambda n}$ reaching it at some finite instant $s_-<0$.
		\item[3.] The points $(x_+,0)$ and $(x_-,0)$ are different. In fact, $x_+<x_-$.
	\end{itemize}
\end{pro}

\begin{proof}
\textit{1.} Arguing by contradiction, suppose that $\gamma_+(s)\subset\Lambda_1,\ \forall s>0$. Recall that $\gamma_+(0)=(0,1)$ and $\gamma_+(s)\subset\Lambda_1$ for $s>0$ small enough, hence the monotonicity properties of $\Lambda_1$ ensure that $\gamma_+$ can be expressed as a graph $y=f(x)$ with $f(x)$ satisfying $f(0)=0$ and $f'(x)<0$, for $x>0$ small enough. Consequently, since the orbits are proper curves in $\Theta_1$, $\gamma_+$ would be globally defined by the graph of $f(x)$ satisfying $f'(x)<0\, \forall x>0$ and $\lim_{x\rightarrow\infty}f(x)=c_0\geq 0$. Thus, the curve $\alpha_+(s)=(x_+(s),z_+(s))$ generated by $\gamma_+$ has positive geodesic curvature with $x_+'(s)>0,\ \forall s>0$ (since $\gamma_+$ lies over the axis $x'=y=0$). 

In this way, the $\lH$ $\sig_+$ generated by rotating $\alpha_+$ around the $x\n1$-axis is a strictly convex, entire graph over $\R^n$, whose mean curvature function is $H_{\sig_+}(p)=\langle\eta_p,e\n1\rangle+\lambda$ at each $p\in\sig_+$. Since $\lambda>1$, there exists a positive constant $H_0\in\R$ such that $H_{\sig_+}> H_0>0$. From here, as we can find a tangent point of intersection between the sphere $\S^n(1/H_0)$ of constant mean curvature equal to $H_0$ and $\sig_+$ in such a way that their unit normals agree and $\S^n(1/H_0)$ lies above $\sig_+$, the mean curvature comparison principle leads a contradiction. 

\textit{2.} The same argument for the orbit $\gamma_-(s)$ carries over verbatim, that is, $\gamma_-$ cannot stay forever in $\Lambda_2$ and it converges to a point $(x_-,0)$ with $x_-\geq\frac{n-1}{\lambda n}$, being either $e_0$ with $s\rightarrow-\infty$, or a finite point reaching it at some finite instant $s_-<0$. Now, it remains to prove that $(x_-,0)$ cannot be the equilibrium point $e_0=\big(\frac{n-1}{\lambda n},0\big)$. To this end, note that $\gamma_-$ cannot intersect the curve $\Gamma_1$ because of the monotonicity properties of $\Lambda_2$, and the horizontal graph $\Gamma_1(y)$ given by \eqref{curvagamma} achieves a global maximum at $y_0=-1/\lambda$, and so $\Gamma_1(y_0)>\frac{n-1}{\lambda n}=\Gamma_1(0)$.  Thus, when $\gamma_-$ leaves the maximum of $\Gamma_1$ at his left-hand side, $\gamma_-$ cannot go backwards and converge to $e_0$, since it would contradict the monotonicity of $\Lambda_2$. See Figure \ref{contradiccionybien} left, the pointed plot of the orbit $\gamma_-$. 

\textit{3.} First we prove that $x_+\neq x_-$. Arguing by contradiction, suppose that $x_+=x_-:=\hat{x}$. Note that $(\hat{x},0)\neq e_0$ since we discussed in item \textit{2} that $(x_-,0)\neq e_0$. In this situation the orbits $\gamma_+$ and $\gamma_-$ meet each other orthogonally at $(\hat{x},0)$ (see Figure \ref{contradiccionybien} left, the continuous plot of $\gamma_+$ and $\gamma_-$). By uniqueness of the Cauchy problem they can be smoothly glued together to form a larger orbit $\gamma_0$ satisfying the following: $\gamma_0$ is a compact arc joining the points $(0,1)$ and $(0,-1)$, strictly contained in $\Lambda_1\cup\Lambda_2\cup\{(\hat{x},0)\}$. Hence, the rotational $\lH$ generated by this orbit would be a simply connected, closed hypersurface, i.e. a rotational sphere, but this fact contradicts Lemma \ref{noclosed}. 

To finish, we check that $x_+<x_-$ by another contradiction argument. Indeed, suppose that $x_+>x_-$ and let us focus on the orbit $\gamma_-$. We will keep track of $\gamma_-(s)$ by moving within it with the parameter $s$ decreasing; recall that $\gamma_-(s)$ tends to $(0,-1)$ as the parameter $s$ increases. In this setting, the orbit $\gamma_-$ would be at the left-hand side of the orbit $\gamma_+$ when they intersect the axis $y=0$. As $\gamma_+$ and $\gamma_-$ cannot intersect each other and by properness of the orbits in $\Theta_1$, the only possibility is that $\gamma_-$ enters the region $\Lambda_2$ and later $\Lambda_4$ at some finite instant. By properness, monotonicity and since $\gamma_-$ cannot converge to the segment $\{(0,y),\ |y|<1\}$, as it was mentioned in Section $3$, $\gamma_-$ cannot do anything but enter the region $\Lambda_3$. As $\gamma_-$ cannot self-intersect, it follows that $\gamma_-$ ends up converging asymptotically to $e_0$ (Figure \ref{contradiccionybien} left, the dashed plot of the orbit $\gamma_-$). But this is a contradiction with the fact that $e_0$ is asymptotically stable and with motion of the orbit $\gamma_-$, since it tends to \emph{escape} from $e_0$ as $s$ increases. So, the only possibility is that $\gamma_+$ is at the left-hand side of $\gamma_-$ when they converge to the axis $y=0$, either converging to $e_0$ (Figure \ref{contradiccionybien} right, dashed plot) or intersecting the axis $y=0$ at a finite point $(x_+,0)$ (Figure \ref{contradiccionybien} right, continuous plot).
\end{proof}

\begin{figure}[H]
	\centering
	\includegraphics[width=.9\textwidth]{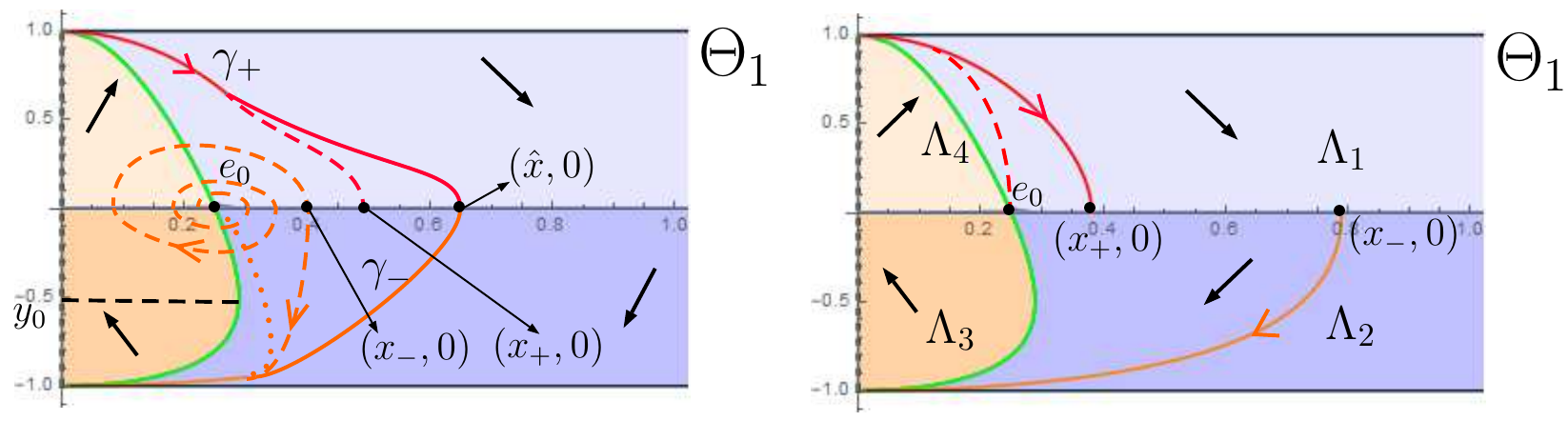}
	\caption{Left: the configurations that cannot happen in $\Theta_1$ for $\gamma_+$ and $\gamma_-$. Right: the configuration of the orbits $\gamma_+$ and $\gamma_-$ in $\Theta_1$ when reaching the axis $y=0$.}
	\label{contradiccionybien}
\end{figure}

As seen on the right-hand side of Figure \ref{contradiccionybien}, we get a first approximation about how to represent properly the orbits $\gamma_+$ and $\gamma_-$ when they intersect the axis $y=0$. However, we must carry on analyzing the global behavior of $\gamma_+$ and $\gamma_-$ and its corresponding generated curves $\alpha_+$ and $\alpha_-$.

On the one hand, if $\gamma_+$ intersects the axis $y=0$ at a finite point $(x_+,0)$ different to the equilibrium $e_0$, then $\gamma_+$ enters the region $\Lambda_2$ but cannot intersect $\gamma_-$, and so $\gamma_+$ has to enter the region $\Lambda_3$. By monotonicity, properness and since $\gamma_+$ cannot converge to the segment $\{(0,y),\ |y|<1\}\subset\Theta_1$, the only possibility is that $\gamma_+$ has to enter the region $\Lambda_4$. As $\gamma_+$ cannot self-intersect, we see that $\gamma_+$ ends up converging asymptotically to $e_0$ (see Figure \ref{lmayor1faseseje}, left). In any case, this orbit generates a complete, arc-length parametrized curve $\alpha_+(s)=(x_+(s),z_+(s))$ with the following properties: the $x_+(s)$-coordinate is bounded and converges to the value $\frac{n-1}{\lambda n}$, that is, $\alpha_+(s)$ converges to the straight line $x=\frac{n-1}{\lambda n}$ for $s\rightarrow\infty$; and the $z_+(s)$-coordinate is strictly increasing since $\gamma_+\subset\Theta_1$ and so $z_+'(s)>0$, which implies that $\alpha_+(s)$ has no self-intersections, i.e. is an embedded curve. 

Hence, the hypersurface $\sig_+$ generated after rotating $\alpha_+$ around the $x\n1$-axis, is a properly embedded, simply connected $\lH$ that converges to the CMC cylinder $C(\frac{n-1}{\lambda n})$ of radius $\frac{n-1}{\lambda n}$. To be more specific:
\begin{itemize}
\item if $\lambda>\sqrt{n-1}/2$, then $\gamma_+$ converges to $e_0$ spiraling around it infinitely many times. This implies that $\alpha_+$ intersects the line $x=\frac{n-1}{\lambda n}$ infinitely many times, and so does $\sig_+$ with $C(\frac{n-1}{\lambda n})$. See Figure \ref{lmayor1faseseje} left and right, the continuous plot. 
\item if $\lambda<\sqrt{n-1}/2$, then $\gamma_+$ converges to $e_0$ \emph{directly}, that is without spiraling around it. As a consequence, $\alpha_+'$ is never vertical and thus $\sig_+$ is a strictly convex graph that converges to $C\left(\frac{n-1}{n}\right)$. See Figure \ref{lmayor1faseseje} left and right, the dashed plot.
\item if $\lambda=\sqrt{n-1}/2$, then $\gamma_+$ converges to $e_0$ after spiraling around it a finite number of times, and so $\sig_+$ is a graph outside a compact set.
\end{itemize}


\begin{figure}[H]
	\centering
	\includegraphics[width=.9\textwidth]{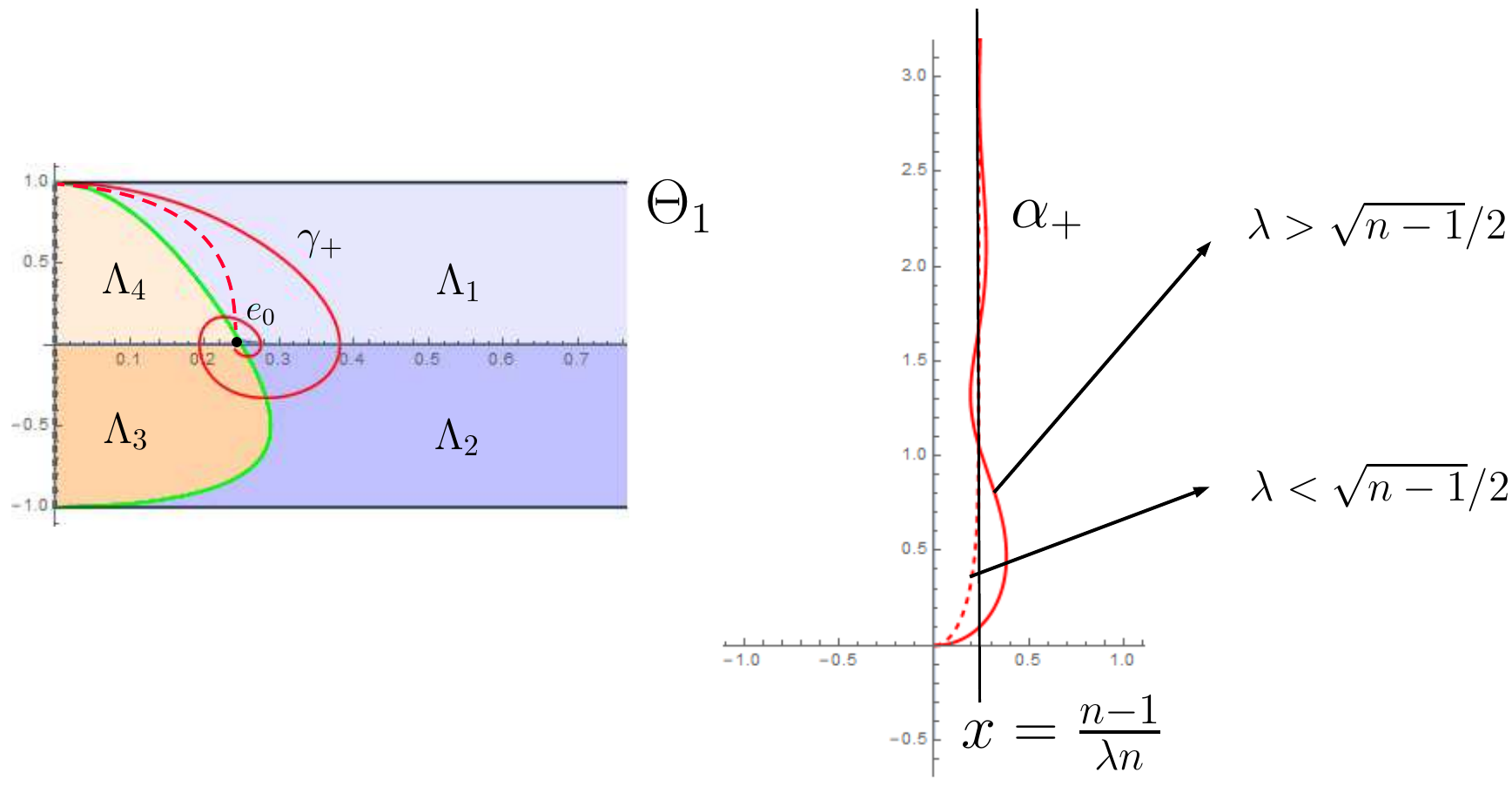}
    \caption{Left: the phase plane $\Theta_1$ and the possible orbits $\gamma_+$. Right: the corresponding arc-length parametrized curves $\alpha_+$.}
	\label{lmayor1faseseje}
\end{figure}

On the other hand, recall that $\gamma_-$ intersects the axis $y=0$ at some finite point $\gamma_-(s_-)=(x_-,0),\ s_-<0$, lying on the right-hand side of $e_0$. Decreasing $s<s_-$ we get that $\gamma_-$ enters the region $\Lambda_1$. By monotonicity, properness and since $\gamma_+$ and $\gamma_-$ cannot intersect in $\Theta_1$, the only possibility for $\gamma_-$ is to have as endpoint some $\gamma_-(s_1)=(x_1,1)$ with $x_1>0$ and $s_1<s_-$ (see Figure \ref{lmayor1faseseje2}, top left). At this instant we have $x_-(s_1)=x_1$ and $x'_-(s_1)=1$, and ODE \eqref{odemedia} ensures us that $z''_-(s_1)>0$, that is the height of $\alpha_-$ reaches a minimum. As a consequence, for $s<s_1$ close enough to $s_1$ the height function $z_-(s)$ is decreasing, i.e. $z'_-(s)<0$ and thus $\alpha_-(s)$ generates an orbit which is contained in $\Theta_{-1}$; now, $\varepsilon=-1$ which agrees with the sign of $z'_-(s)$. For the sake of clarity, we will keep naming $\gamma_-$ to this orbit in $\Theta_{-1}$.

In this situation, $\gamma_-\subset\Theta_{-1}$ is an orbit with $\gamma_-(s_1)=(x_1,1)$ as endpoint and lying in the region $\Lambda_+$. Again, by monotonictiy and properness the orbit $\gamma_-$ has to intersect the axis $y=0$ in an orthogonal way, and then enter the region $\Lambda_-$. Lastly, Proposition \ref{contradicecomparacioncurvaturamedia} ensures us that $\gamma_-$ cannot stay contained in $\Lambda_-$ with the $x_-(s)$-coordinate tending to infinity, hence $\gamma_-$ intersects the line $y=-1$ at some $\gamma(s_2)=(x_2,-1),\ s_2<s_1$ (see Figure \ref{lmayor1faseseje2}, bottom left).

Again, in virtue of Equation \eqref{odemedia}, at the instant $s=s_2$ the height function $z_-(s)$ of $\alpha_-$ satisfies $z''_-(s_2)<0$, and so $z_-(s)$ achieves a maximum at $s=s_2$ and thus $z_-(s)$ for $s<s_2$ close enough to $s_2$ is an increasing function, and so $\alpha_-(s)$ for $s<s_2$ close enough to $s_2$ generates an orbit in $\Theta_1$, which will be still named $\gamma_-$. Now, $\gamma_-$ starts at the point $(x_2,-1)$ and by monotonicty and properness it has to go from $\Lambda_2$ to $\Lambda_1$ as $s<s_2$ decreases. Since $\gamma_-$ cannot self-intersect, we get that $\gamma_-$ has to reach again the line $y=1$ at some point $(x_3,1)$, with $x_3>x_1$ (see again Figure \ref{lmayor1faseseje2}, top left).

This process is repeated and we get a complete, arc-length parametrized curve $\alpha_-(s)$ with self-intersections and whose height function increases and decreases until reaching the $x\n1$-axis orthogonally (see Figure \ref{lmayor1faseseje2}, right). Therefore, the $\lH$ obtained by rotating $\alpha_-$ is properly immersed (with self-intersections) and simply connected.


\begin{figure}[H]
	\centering
	\includegraphics[width=.7	\textwidth]{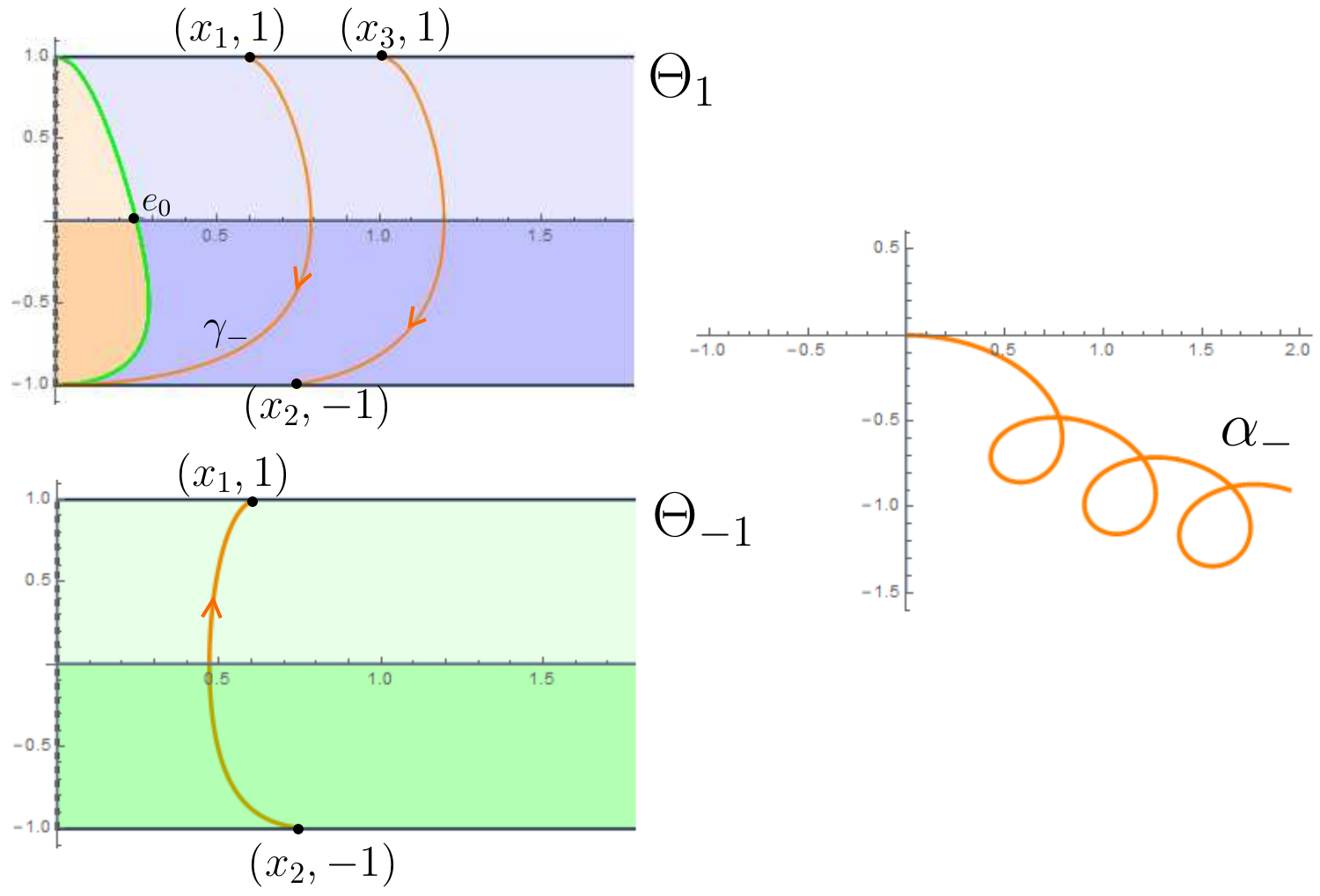}
    \caption{Left: the phase planes $\Theta_1$ and $\Theta_{-1}$ and the orbit $\gamma_-$. Right: the corresponding arc-length parametrized curve $\alpha_-$.}
	\label{lmayor1faseseje2}
\end{figure}
\vspace{-0.2cm}

Our second goal concerns the classification of complete $\lHn$ non-intersecting the axis of rotation. For that, let us take $r_0>0$ and $\gamma(s)$ the orbit in $\Theta_1$ passing through the point $(r_0,0)$ at the instant $s=0$. Then, $\gamma$ is an arc having one endpoint of the form $(r_1,1),\ r_1>0$\footnote{We can suppose that $r_1>0$, since if $r_1=0$ then $\gamma$ is the orbit corresponding to the $\lH$ intersecting the axis of rotation, already described in Figure \ref{lmayor1faseseje}.}, and either converges to $e_0$ as $s\rightarrow\infty$ or has another endpoint of the form $(r_2,-1)$. In the second case, the orbit $\gamma$ continues in $\Theta_{-1}$ as a compact arc and then goes in again in $\Theta_1$. By propernes, after a finite number of iterations, the orbit $\gamma$ eventually converges to $e_0$ (see Figure \ref{lmayor1fueraeje}, left).

This configuration ensures us that the $\lHn$ associated to $\gamma$ is properly immersed and diffeomorphic to $\S^{n-1}\times\R$, with one end converging to $C\left(\frac{n-1}{\lambda n}\right)$ and the other end having unbounded distance to the axis of rotation, looping and self-intersecting infinitely many times (see Figure \ref{lmayor1fueraeje}, right).

\begin{figure}[H]
	\centering
	\includegraphics[width=.7\textwidth]{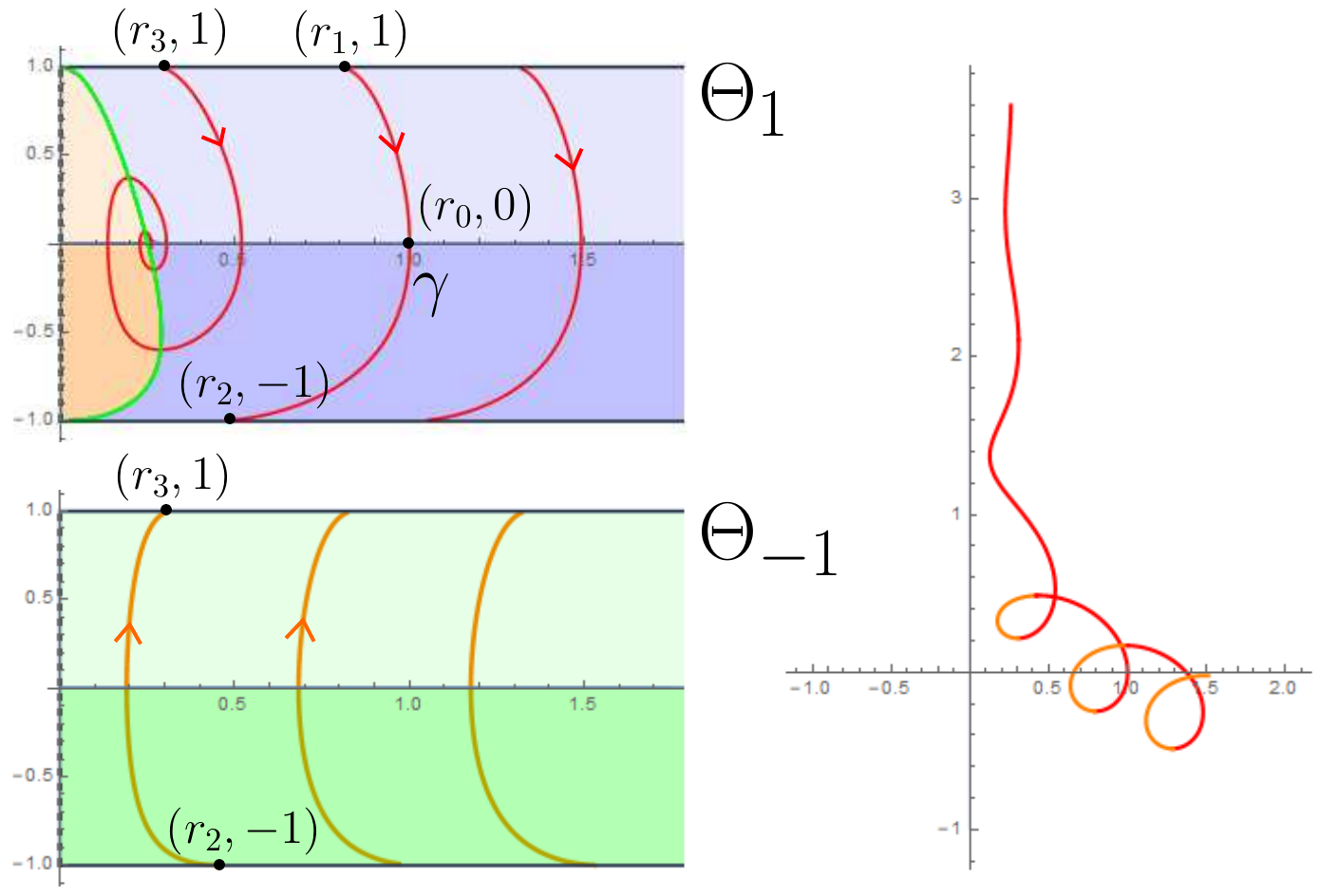}
	 \caption{Left: the phase planes $\Theta_1$ and $\Theta_{-1}$ and the orbit $\gamma$. Right: the corresponding arc-length parametrized curve $\alpha$.}
	\label{lmayor1fueraeje}
\end{figure}

\begin{center}
{\Large \textbf{\underline{Case $\lambda=1$}}}
\end{center}
\vspace{-.25cm}

Now we suppose that $\lambda=1$. In this situation, the curve $\Gamma_1$ given by Equation \eqref{curvagamma} for $\varepsilon=1$ is a connected arc in $\Theta_1$ having the point $(0,1)$ as endpoint, and the line $y=-1$ as an asymptote. Thus, $\Theta_1$ has four monotonicity regions, $\Lambda_1,...,\Lambda_4$ (see Figure \ref{fasesligual1}, left). The region $\Lambda_1\cup\Lambda_2$ corresponds to points with positive geodesic curvature, while the region $\Lambda_3\cup\Lambda_4$ corresponds to points with negative geodesic curvature. For $\varepsilon=-1$, the curve $\Gamma_{-1}$ in $\Theta_{-1}$ is empty, and there are only two monotonicity regions $\Lambda_+$ and $\Lambda_-$ (see Figure \ref{fasesligual1}, right).

\begin{figure}[H]
\centering
\includegraphics[width=.9\textwidth]{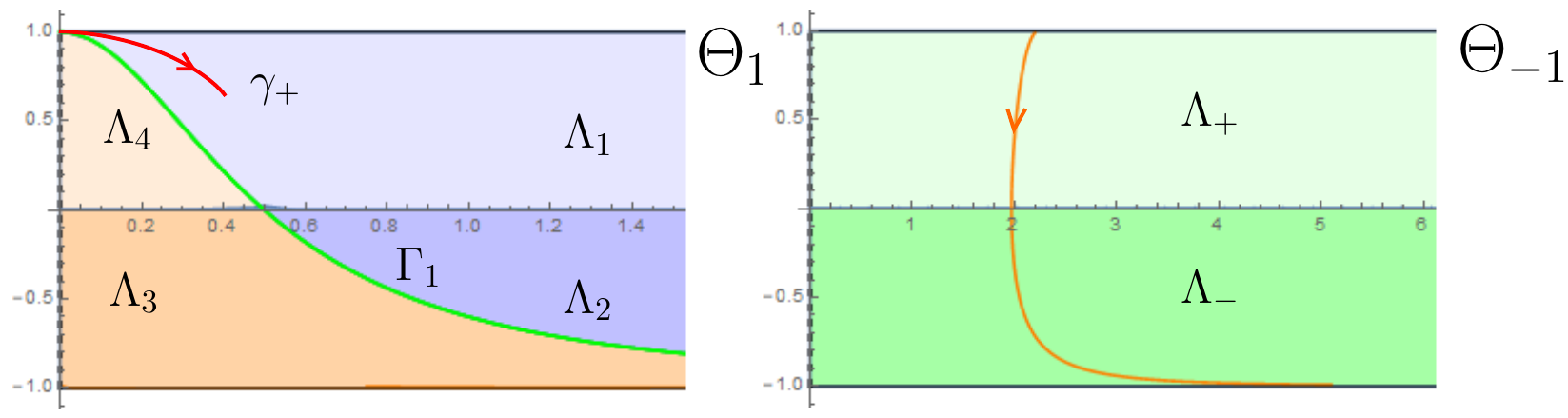}
\caption{The phase planes $\Theta_\varepsilon,\ \varepsilon=\pm1$ for $\lambda=1$, their monotonicity regions and two orbits following the motion at each monotonicity region.}
\label{fasesligual1}
\end{figure}

We first study the rotational $\H_1$-hypersurfaces intersecting the axis of rotation. For this purpose, we must begin by pointing out that a horizontal hyperplane $\Pi$=$\{x\n1=c_0,\ c_0\in\R\}\subset\R^{n+1}$ oriented with unit normal $\eta=-e\n1$ is precisely an example of such an $\H_1$-hypersurface. Indeed, the mean curvature of $\Pi$ is identically zero, and Equation \eqref{defilambdasup} for the density vector $v=e\n1$ is
$$
H_\Pi=\langle\eta,e\n1\rangle+\lambda=\langle-e\n1,e\n1\rangle+1=0.
$$
This fact, along with the uniqueness of the Cauchy problem associated to \eqref{1ordersys} implies that any orbit $\gamma\in\Theta_\varepsilon$ cannot have a limit point in the line $y=-1$, since these points correspond to orbits that generate horizontal hyperplanes with downwards orientation.

Now, with the aim of looking for the remaining $\H_1$-hypersurfaces intersecting the axis of rotation, we follow the same procedure than the one used for the case $\lambda>1$. Note that by Lemma \ref{existorbitaext} there exists a unique orbit $\gamma_+(s)$ in $\Theta_1$ with $\gamma_+(0)=(0,1)$ generating an arc-length parametrized curve $\alpha_+$ intersecting the axis of rotation at the instant $s=0$ and with $\kappa_{\alpha_+}(s)>0$ for $s>0$ small enough. Again, item \textit{1.} in Proposition \ref{contradicecomparacioncurvaturamedia} ensures us that: either $\gamma_+$ converge directly to $e_0$ with $s\rightarrow\infty$; or $\gamma_+$ intersects the axis $y=0$ at a point $(x_+,0)$ with $x_+>\frac{n-1}{n}$ at some finite instant. In this latter case, $\gamma_+$ enters the region $\Lambda_2$ and by monotonicity and properness, $\gamma_+$ intersects the curve $\Gamma_1$ and then enters the region $\Lambda_3$. Since $\gamma_+$ cannot converge to a point $(0,y),\ |y|<1$, $\gamma_+$ has to enter the region $\Lambda_4$, and lastly $\gamma_+$ intersects the curve $\Gamma_1$ entering again the region $\Lambda_1$. Finally, since $\gamma_+$ cannot self-intersect, we see that $\gamma_+$ has to converge asymptotically to $e_0$. Specifically:

\begin{itemize}
\item if $n=2,3,4$, then $1>\sqrt{n-1}/2$ and $\gamma_+$ spirals around $e_0$ an infinite number of times. 
\item if $n=5$, then $1=\sqrt{n-1}/2$ and $\gamma_+$ converges to $e_0$ after spiraling around it a finite number of times. 
\item if $n\geq 6$, then $1<\sqrt{n-1}/2$ and $\gamma_+$ converges directly to $e_0$, without looping around it.
\end{itemize}
Hence, in any case, the $\H_1$-hypersurface $\sig_+$ obtained by rotating $\alpha_+$ around the $x\n1$-axis is a complete, properly embedded and simply connected hypersurface that converges to the CMC cylinder $C(\frac{n-1}{n})$ (see the right-hand side of Figure \ref{lmayor1faseseje} since it is a similar case). 

Secondly, we describe rotational $\H_1$-hypersurfaces non-intersecting the axis of rotation. To do so, we first analyze the behavior of the orbits in $\Theta_1$. Let us fix $\hat{x}>0$, and consider the orbit $\gamma(s)$ in $\Theta_1$ such that $\gamma(0)=(\hat{x},0)$. Moreover, we can suppose that $\gamma\neq\gamma_+$. For $s>0$, the monotonicity properties of $\Theta_1$ ensure us that $\gamma(s)$ converge asymptotically to $e_0$, but $\gamma$ and $\gamma_+$ cannot intersect each other, and so $\gamma(s)$ unwraps from $e_0$ a finite number of times for $s<0$. Consequently, $\gamma(s)$ intersects the axis $y=0$ a finite number of times for $s<0$, and so we can denote $(x_0,0)$ to the last intersection of $\gamma$ with $y=0$. 

Now, we claim that $(x_0,0)$ is on the right-hand side of $e_0$. Arguing by contradiction, suppose that $(x_0,0)$ is on the left-hand side of $e_0$ (see Figure \ref{ligual1fases}, top left, blue orbit,  to clarify this proof). Then, the orbit $\gamma(s)$ cannot intersect the curve $\Gamma_1$; otherwise, $\gamma$ would intersect $y=0$ again by monotonicity of $\Lambda_2$. So, by properness and since $\gamma$ cannot have an endpoint at $y=-1$, the only possibility for $\gamma(s)$ is to converge to the line $y=-1$. As a consequence, $\gamma$ can be locally expressed as a graph $(x,h(x))$ with $h(x_0)=0,\ h'(x)<0,\ \forall x>x_0$ and $h(x)\rightarrow -1$ when $x\rightarrow\infty$. 

To get the contradiction, we compare the orbits of the associated systems \eqref{1ordersys} of rotational hypersurfaces of two different prescribed mean curvature. Firstly, we remind that $\H_\lambda$-hypersurfaces arises as a particular case when in Equation \eqref{prescribedMC} we prescribe the function $\H_\lambda(z)=\langle z,e\n1\rangle+\lambda,\ \forall z\in\S^n$. Now, consider the function $\f(z)=1/2\cos(\pi/2\langle z,e\n1\rangle),\ \forall z\in\sn$, which is a non-negative, even function in $\S^n$ and such that $\f(\pm e\n1)=0$, and as detailed in \cite{BGM2}, we can also study the rotational $\f$-hypersurfaces by just substituting the prescribed function $\f(y)=1/2\cos(\pi/2y)$ in system \eqref{1ordersys} instead of $y+\lambda$. The study made in Sections 2 and 4 in \cite{BGM2} ensures us that the orbits for the prescribed function $\f$ are closed curves, symmetric with respect to the axis $y=0$ and that never intersect the lines $y=\pm 1$. For this prescribed function we view its orbits $\sigma_\f(t)=(x_\f(t),y_\f(t))$ in the phase plane $\Theta_1$ of system \eqref{1ordersys}. Suppose that there are instants $s_0,t_0$ such that $\sigma_\f(t_0)=\gamma(s_0)$. Then, since $\f(y)\leq 1+y$, with equality if and only if $y=-1$, a standard comparison of ODE's yields that $y'(s_0)<y'_\f(t_0)$. At this point, we take $0<x_0^*<x_0$ and $\sigma_\f$ such that $\sigma_\f(0)=(x_0^*,0)$. This orbit $\sigma_\f$ can be also expressed as a graph $(x,f(x))$ such that $f(x_0^*)=0$, $f(x)$ decreases until reaching a minimum and then $f$ increases intersecting again the axis $y=0$. By continuity, there exists some $x_*>x_0$ such that $f(x_*)=h(x_*)$. Therefore, there exist $s_*,t_*<0$ such that $\gamma(s_*)=\sigma_\f(t_*)$, where their second coordinates would satisfy $y'(s_*)>y'_\f(t_*)$ (see Figure \ref{ligual1fases}, top left), arriving to the expected contradiction. 

Since $(x_0,0)$ is on the right-hand side of $e_0$, $\gamma(s)$ has to intersect $\Gamma_1$ at some instant $s_0<0$ and enter the region $\Lambda_2$. Now, monotonicity and properness allows us to ensure that $\gamma(s)$ reaches the line $y=1$ at some finite point $\gamma(s_1)=(x_1,1),\ s_1<0$, with $x_1>0$ (see Figure \ref{ligual1fases}, top right). Consequently, the arc-length parametrized curve $\alpha(s)=(x(s),z(s))$ associated to this orbit $\gamma$ satisfies $x(s_1)=x_1,\ x'(s_1)=1$ and for $s>s_1$ the $x(s)$-coordinate ends up converging to the value $\frac{n-1}{n}$, that is $\alpha(s)$ converges to the line $x=\frac{n-1}{n}$ as $s\rightarrow\infty$. The $z(s)$-coordinate is strictly increasing, since $\mathrm{sign}(z'(s))=\varepsilon=1$.

To finish, note that the behavior of the orbit $\gamma$ in $\Theta_{-1}$ follows easily from the monotonicity properties. This orbit $\gamma$ has to intersect orthogonally the axis $y=0$ and then converge to the line $y=-1$ (see Figure \ref{ligual1fases}, bottom left). Note that $\gamma$ cannot converge to some line $\{y=y_0,\ y_0\in (-1,0)\}$ by using the same reasoning that the one contained in the proof of item \textit{1} in Proposition \ref{contradicecomparacioncurvaturamedia}. In this situation, the $x(s)$-coordinate of $\alpha$ is unbounded as $s\rightarrow -\infty$ and $z(s)$ is a strictly decreasing function, reaching its minimum at the instant $s_1$.

The $\H_1$-hypersurface generated by rotating $\alpha$ around the $x\n1$-axis is complete, properly immersed and diffeomorphic to $\S^{n-1}\times\R$, with one end converging to the CMC cylinder $C\left(\frac{n-1}{n}\right)$ and the other end being a graph outside a ball in $\R^n$. Note that every such $\H_1$-hypersurface has a self-intersection, hence it is not embedded (see Figure \ref{ligual1fases}, bottom right).

\begin{figure}[H]
	\centering
	\includegraphics[width=.8\textwidth]{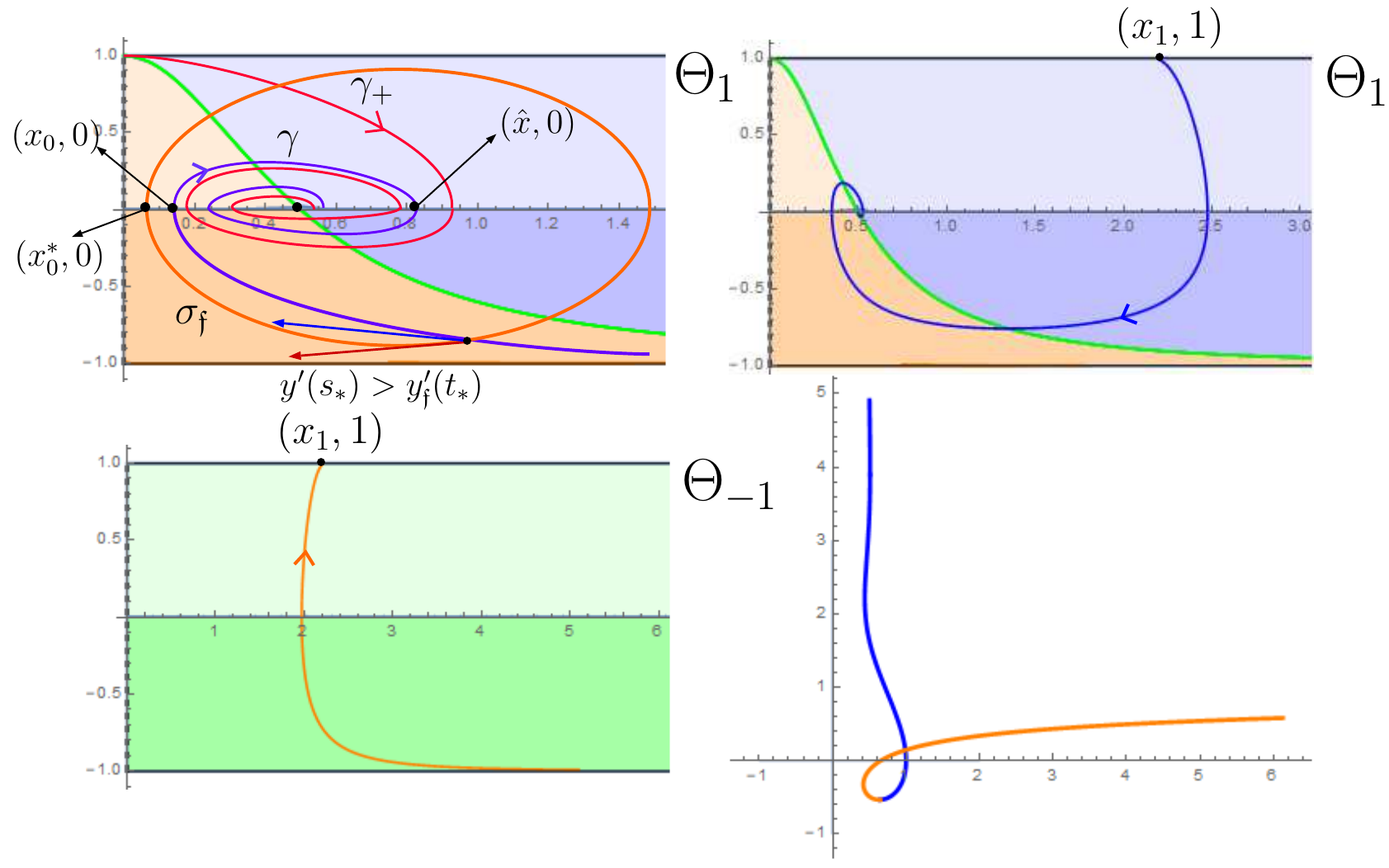}
	\caption{Top left: the configuration that cannot happen in $\Theta_1$ for $\gamma$. Top right and bottom left: the phase planes $\Theta_1$ and $\Theta_{-1}$ and the orbit $\gamma$. Bottom right: the corresponding arc-length parametrized curve $\alpha$.}
	\label{ligual1fases}
\end{figure}

\begin{center}
{\Large \textbf{\underline{Case $\lambda<1$}}}
\end{center}
\vspace{-.25cm}

Finally, we consider the case when $0<\lambda<1$. In this situation, for $\varepsilon=1$, the curve $\Gamma_1$ given by Equation \eqref{curvagamma} is a connected arc in $\Theta_1$ having the point $(0,1)$ as endpoint, and an asymptote at the line $y=-\lambda$. Consequently, in $\Theta_1$ there are four monotonicity regions called $\Lambda_1^+,\dots,\Lambda_4^+$ (see Figure \ref{lmenor1todo}, top left). For $\varepsilon=-1$, the curve $\Gamma_{-1}$ in $\Theta_{-1}$ is also a connected arc with $(0,-1)$ as endpoint and an asymptote also at the line $y=-\lambda$, then there are three regions of monotony denoted by $\Lambda_1^-,\Lambda_2^- \; \text{and} \; \Lambda_3^-$ (see Figure \ref{lmenor1todo}, bottom left).

Once again, we begin describing the $\lHn$ intersecting orthogonally the axis of rotation. On the one hand, by Lemma \ref{existorbitaext} we know that there exists a unique orbit $\gamma_+(s)$ in $\Theta_1$ with $(0,1)$ as endpoint. By reasoning as done in the previous cases, we can conclude that $\gamma_+$ has to converge asymptotically to $e_0$ (see Figure \ref{lmenor1todo}, top left). Therefore, the $\lH$ $\sig_+$ obtained by rotating $\alpha_+$ around the $x\n1$-axis is a properly embedded, simply connected hypersurface converging asymptotically to the CMC cylinder $C\big(\frac{n-1}{\lambda n}\big)$ (see Figure \ref{lmenor1todo}, right). Additionally, the obtained discussion for $\Sigma_+$ depending on the value of $\lambda$ with respect to $\sqrt{n-1}/2$ is exactly the same than the one that we get in the case $\lambda>1$. On the other hand, Lemma \ref{existorbitaext} allows us to assert that there exists a unique orbit $\gamma_-(s)$ in $\Theta_{-1}$ satisfying $\gamma_-(0)=(0,-1)$. Then $\gamma_-$ belongs to $\Lambda_2^-$ for $s<0$ close enough to $s=0$ (see Figure \ref{lmenor1todo} bottom left). By monotonicity, $\gamma_-$ cannot intersect the curve $\Gamma_{-1}$, and by properness and by Proposition \ref{contradicecomparacioncurvaturamedia}, $\gamma_-(s)$ has to converge to the line $y=-\lambda$ when $s\rightarrow-\infty$. This implies that the $\lH$ $\sig_-$ obtained by rotating $\alpha_-$ around the $x\n1$-axis is an entire, strictly convex graph (see Figure \ref{lmenor1todo}, right).

\begin{figure}[H]
	\centering
	\includegraphics[width=.75\textwidth]{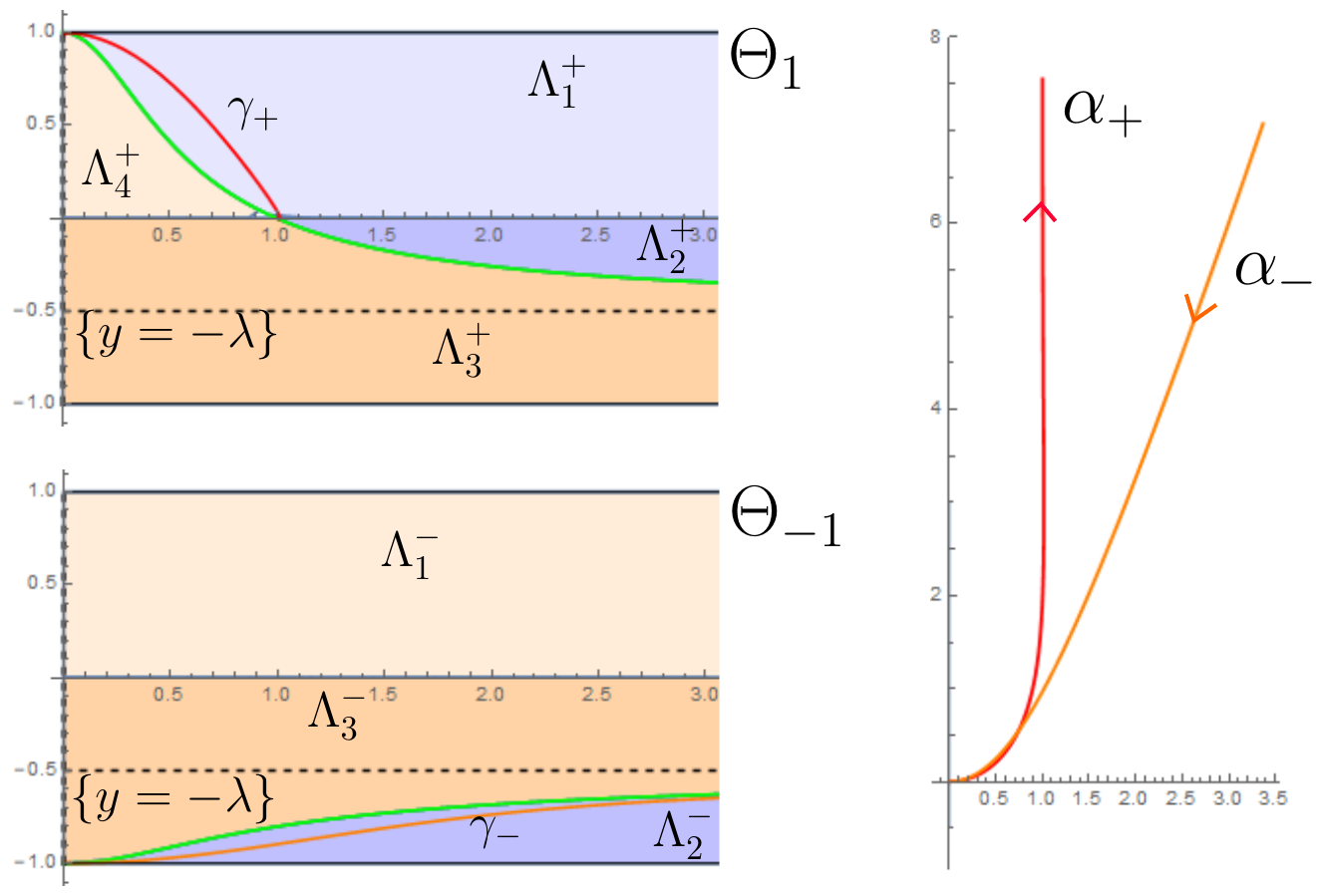}
	\caption{Left: the phase planes $\Theta_1$ and $\Theta_{-1}$ and the orbits $\gamma_+$ and $\gamma_-$. Right: the corresponding arc-length parametrized curves $\alpha_+$ and $\alpha_-$.}
	\label{lmenor1todo}
\end{figure}

Finally, we analyze the $\lHn$ non-intersecting the axis of rotation. For that, let $\gamma$ be an orbit in $\Theta_1$ passing through a point $(\hat{x},0),\ \hat{x}>0$. By monotonicity and properness, $\gamma(s)$ has to converge asymptotically to $e_0$ as $s\rightarrow\infty$, either directly, spiraling around if a finite number of times or infinitely many times. If we decrease the parameter $s$, and noting that $\gamma$ cannot intersect $\gamma_+$, we see that $\gamma$ has to intersect the axis $y=0$ in a last point $(x_0,0)$. Note that without loss of generality we can assume that $\gamma$ reaches the point $(x_0,0)$ at the instant $s=0$, and to conclude the discussión we distinguish two cases: if $(x_0,0)$ lies at the right-hand side or the left-hand side of $e_0=\big(\frac{n-1}{\lambda n},0\big)$.

First, suppose that $x_0<\frac{n-1}{\lambda n}$. Decreasing $s<0$ we see that $\gamma(s)$ cannot intersect $\Gamma_1$, since otherwise it would intersect $y=0$ again, and therefore $\gamma$ stays in $\Lambda_3^+$ until reaching some $(x_1,-1)$ as endpoint (see Figure \ref{lmenor1todofueraeje}, top left, red orbit). Now, the orbit $\gamma$ continues in $\Theta_{-1}$ entering the region $\Lambda_2^-$ and converging to the line $y=-\lambda$ (see Figure \ref{lmenor1todofueraeje}, bottom left, red orbit). If we denote by $\alpha(s)$ to the arc-length parametrized curve generated by $\gamma$ we get that the rotation of $\alpha$ around the $x\n1$-axis gives us a properly embedded $\lH$, diffeomorphic to $\S^{n-1}\times\R$ with two ends; one converging to $C\big(\frac{n-1}{\lambda n}\big)$ and the other being a strictly convex graph (see Figure \ref{lmenor1todofueraeje}, center).

Now, suppose that $x_0>\frac{n-1}{\lambda n}$. Decreasing $s<0$, and because $\gamma$ and $\gamma_+$ cannot intersect each other, we see that $\gamma(s)$ stays in $\Lambda_1^+$ until reaching some $(x_2,1)$ as endpoint (see Figure \ref{lmenor1todofueraeje}, top left, orange orbit). Now, $\gamma$ continues in $\Theta_{-1}$ entering the region $\Lambda_1^-$ and then going into $\Lambda_3^-$ after intersecting orthogonally the axis $y=0$. As $\gamma$ cannot stay contained in $\Lambda_3^-$ in virtue of Proposition \ref{contradicecomparacioncurvaturamedia}, we get that $\gamma(s)$ has to enter $\Lambda_2^-$ and converge to the line $y=-\lambda$ when $s\rightarrow-\infty$ (see Figure \ref{lmenor1todofueraeje}, bottom left, orange orbit). Hence, the rotational $\lH$ obtained is properly immersed, diffeomorphic to $\S^{n-1}\times\R$ and with two embedded ends; one converging to $C\big(\frac{n-1}{\lambda n}\big)$ and the other being a strictly convex graph (see Figure \ref{lmenor1todofueraeje}, right).

\begin{figure}[H]
	\centering
	\includegraphics[width=.9\textwidth]{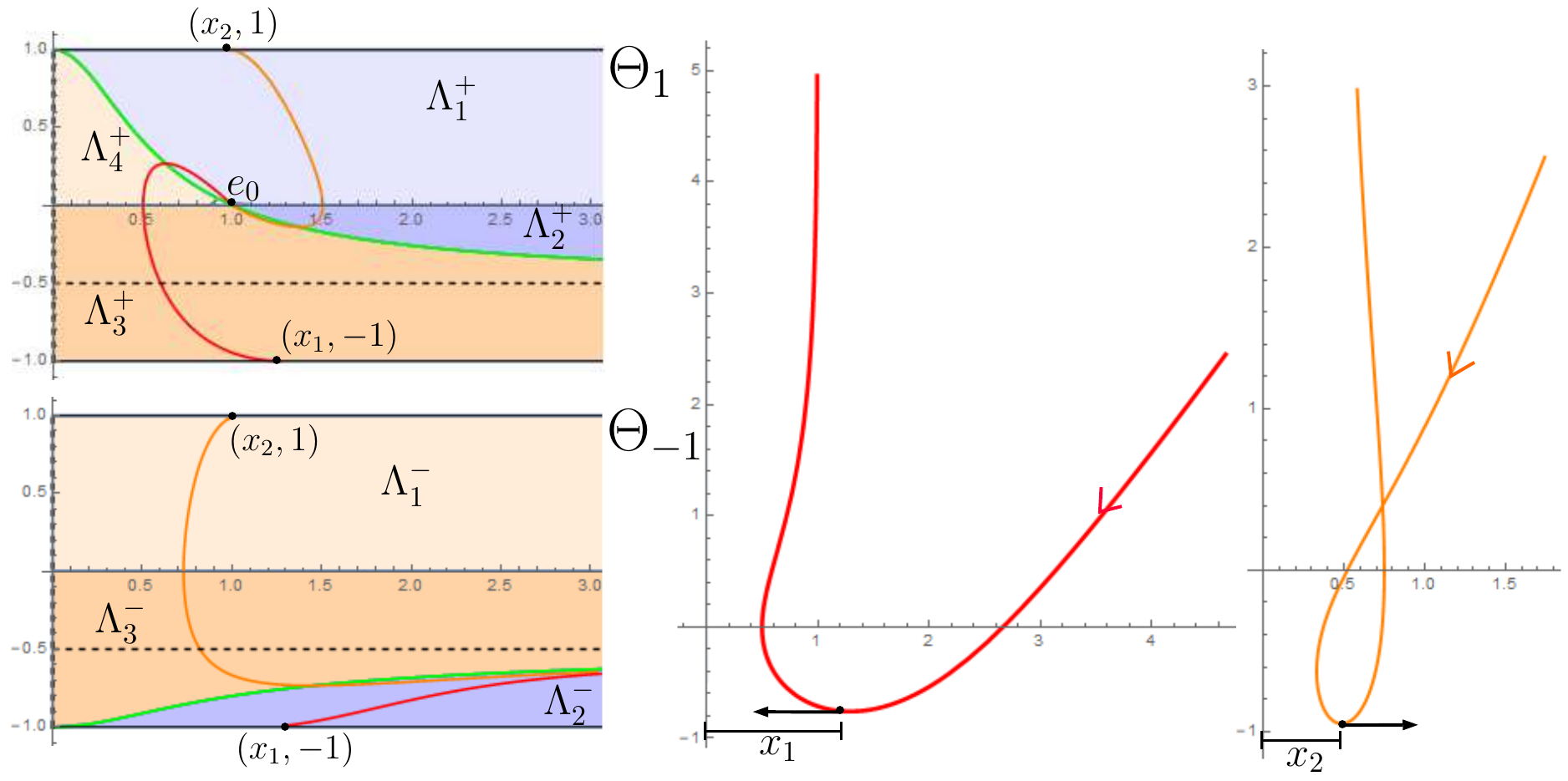}
	\caption{Left: The phase planes $\Theta_1$ and $\Theta_{-1}$ and the two possible configurations for the orbit $\gamma$. Center and right: the two corresponding arc-length parametrized curves $\alpha$.}
	\label{lmenor1todofueraeje}
\end{figure}

To finish, we summarize the discussion carried on along this section in two classification results of the rotational $\lHn$: the first result for the ones intersecting the axis $x_{n+1}$, and the second one for the opposite case. For the very particular case that $n=2$, these results agree with the ones obtained in \cite{Lop}.

\begin{teo}

Let be $\Sigma_+$ and $\Sigma_-$ the complete, rotational $\lHn$ intersecting the axis $x_{n+1}$ with upwards and downwards orientation respectively. Then: 
\begin{itemize}
	\item[1.] For any $\lambda>0$, $\Sigma_+$ is properly embedded, simply connected and converges to the CMC cylinder $C(\frac{n-1}{\lambda n})$ of radius $\frac{n-1}{\lambda n}$. Moreover:
	
	\begin{itemize}
		\item[1.1.] If $\lambda>\sqrt{n-1}/2$, $\sig_+$ intersects $C(\frac{n-1}{\lambda n})$ infinitely many times.
		
		\item[1.2.] If $\lambda=\sqrt{n-1}/2$, $\sig_+$ intersects $C(\frac{n-1}{\lambda n})$ a finite number of times and is a graph outside a compact set.
		
		\item[1.3.] If $\lambda<\sqrt{n-1}/2$, $\sig_+$ is a proper graph over the disk of radius $\frac{n-1}{\lambda n}$.
	\end{itemize}
	
	\item[2.] For $\lambda>1$, $\sig_-$ is properly immersed (with infinitely-many self-intersections), simply connected and has unbounded distance to the axis $x\n1$.
	
	\item[3.] For $\lambda=1$, $\sig_-$ is a horizontal hyperplane.
	
	\item[4.] For $\lambda<1$, $\sig_-$ is a strictly convex, entire graph.
\end{itemize}

\label{Classification1}	
\end{teo}

\begin{teo}
Let $\sig$ be a complete, rotational $\lH$ non-intersecting the axis $x\n1$. Then, $\sig$ is properly immersed and diffeomorphic to $\S^{n-1}\times\R$. One end converges to the CMC cylinder $C(\frac{n-1}{\lambda n})$ of radius $\frac{n-1}{\lambda n}$, and:

\begin{itemize}
	
	\item[1.] If $\lambda>1$, the other end has infinitely-many self-intersections and unbounded distance to the axis $x\n1$.
	
	\item[2.] If $\lambda\leq 1$, the other end is a graph outside a compact set.
\end{itemize}

Moreover, if $\lambda<1$ and the unit normal of $\sig$ at the points with horizontal tangent hyperplane is $-e\n1$, then $\sig$ is embedded.

	\label{Classification2}	
\end{teo}

Observe that the end which converges to $C(\frac{n-1}{\lambda n})$ has the same asymptotic behavior than the one observed in item \textit{1.} in Theorem \ref{Classification1}.

\def\refname{References}

The first author was partially supported by MICINN-FEDER Grant No. MTM2016-80313-P. For the second author, this research is a result of the activity developed within the framework of the Programme in Support of Excellence Groups of the Región de Murcia, Spain, by Fundación Séneca, Science and Technology Agency of the Región de Murcia. Irene Ortiz was partially supported by MICINN/FEDER project PGC2018-097046-B-I00 and Fundación Séneca project 19901/GERM/15, Spain.

\end{document}